\pdfoutput=1
\RequirePackage{ifpdf}
\ifpdf 
\documentclass[pdftex]{sigma}
\else
\documentclass{sigma}
\fi

\numberwithin{equation}{section}

\newtheorem{Theorem}{Theorem}[section]
\newtheorem{Corollary}[Theorem]{Corollary}
\newtheorem{Lemma}[Theorem]{Lemma}
\newtheorem{Proposition}[Theorem]{Proposition}
 { \theoremstyle{definition}
\newtheorem{Remark}[Theorem]{Remark} }

\newcommand\bC{\mathbf C}
\newcommand\bF{\mathbf F}
\newcommand\bH{\mathbf H}
\newcommand\bR{\mathbf R}
\newcommand\bZ{\mathbf Z}

\newcommand\ra{\mathrm a}
\newcommand\rb{\mathrm b}
\newcommand\rc{\mathrm c}
\newcommand\rd{\mathrm d}
\newcommand\re{\mathrm e}
\newcommand\rg{\mathrm g}

\newcommand\rE{\mathrm E}
\newcommand\rF{\mathrm F}
\newcommand\rG{\mathrm G}
\newcommand\rJ{\mathrm J}
\newcommand\rM{\mathrm M}
\newcommand\rO{\mathrm O}
\newcommand\rU{\mathrm U}

\DeclareMathOperator\homology{H}
\renewcommand\H{\homology}

\newcommand\mapsfrom\leftarrow
\newcommand\longto\longrightarrow
\newcommand\mono\hookrightarrow
\newcommand\epi\twoheadrightarrow

\newcommand\<\langle
\renewcommand\>\rangle

\newcommand\Co{\mathrm{Co}}
\newcommand\Fi{\mathrm{Fi}}
\newcommand\Leech{\mathrm{Leech}}
\newcommand\Suz{\mathrm{Suz}}
\newcommand\Perm{\mathrm{Perm}}

\newcommand\Spin{\mathrm{Spin}}

\newcommand\SU{\mathrm{SU}}
\newcommand\SO{\mathrm{SO}}

\DeclareMathOperator\Aut{Aut}
\DeclareMathOperator\Sym{Sym}
\DeclareMathOperator\Alt{Alt}

\DeclareMathOperator\Sq{Sq}

\newcommand\SWeyl{\mathrm{SWeyl}}

\DeclareMathOperator\coker{coker}

\newcommand{\GL}{\mathrm{GL}}
\newcommand{\SL}{\mathrm{SL}}
\newcommand{\Sp}{\mathrm{Sp}}
\newcommand{\Hom}{\mathrm{Hom}}
\newcommand{\Ext}{\mathrm{Ext}}
\DeclareMathOperator{\trace}{trace}
\newcommand\tr\trace

\newcommand{\GSp}{\mathrm{GSp}}

\usepackage{tikz}
\usetikzlibrary{decorations.markings,arrows,decorations.pathreplacing,cd}

\usepackage{arydshln}

\begin{document}
\allowdisplaybreaks

\newcommand{\arXivNumber}{1810.00463}

\renewcommand{\thefootnote}{}

\renewcommand{\PaperNumber}{059}

\FirstPageHeading

\ShortArticleName{Third Homology of some Sporadic Finite Groups}

\ArticleName{Third Homology of some Sporadic Finite Groups\footnote{This paper is a~contribution to the Special Issue on Moonshine and String Theory. The full collection is available at \href{https://www.emis.de/journals/SIGMA/moonshine.html}{https://www.emis.de/journals/SIGMA/moonshine.html}}}

\Author{Theo JOHNSON-FREYD~$^\dag$ and David TREUMANN~$^\ddag$}

\AuthorNameForHeading{T.~Johnson-Freyd and D.~Treumann}

\Address{$^\dag$~Perimeter Institute for Theoretical Physics, Waterloo, Ontario, Canada}
\EmailD{\href{mailto:theojf@pitp.ca}{theojf@pitp.ca}}

\Address{$^\ddag$~Department of Mathematics, Boston College, Boston, Massachusetts, USA}
\EmailD{\href{treumann@bc.edu}{treumann@bc.edu}}

\ArticleDates{Received September 30, 2018, in final form August 06, 2019; Published online August 10, 2019}

\Abstract{We compute the integral third homology of most of the sporadic finite simple groups and of their central extensions.}


\Classification{20D08; 20J06}

\renewcommand{\thefootnote}{\arabic{footnote}}
\setcounter{footnote}{0}

\section{Introduction}

In this paper we compute the third homology of some of the sporadic simple groups, and of their central extensions. For many of these groups we are able to name elements (characteristic classes) that generate $\H^4(G;\bZ)$, the Pontryagin dual of~$\H_3(G)$. In the following table we write~$n.G$ for the Schur covering of the sporadic group~$G$~-- for a sporadic simple group, the covering is always by a cyclic group $n = \H_2(G)$~-- and have left empty spaces where $G = n.G$.
\begin{gather*}
\begin{array}{r|cccccccc}
\hline \hline
& \mathrm{M}_{11} &
\mathrm{M}_{12} &
\mathrm{M}_{22} &
\mathrm{M}_{23} &
\mathrm{M}_{24} &
\\
\hline
n = \H_2(G) & 1 & 2 & 12 & 1 & 1 \\
\H_3(G) & 8 & 2 \times 24 & 1 & 1 & 12 \\
\H_3(n.G) & & 8 \times 24 & 24 & & \\
[3pt]
\hline \hline
\\[-9pt]
& \mathrm{HS} &
\mathrm{J}_2 &
\mathrm{Co}_1 &
\mathrm{Co}_2 &
\mathrm{Co}_3 &
\mathrm{McL} &
\mathrm{Suz} &
\\
\hline
\H_2(G) & 2 & 2 & 2 & 1 & 1 & 3 & 6 \\
\H_3(G) & 2 \times 2 & 30 & 12 & 4 & 6 & 1 & 4 \\
\H_3(n.G) & 2 \times 8 & 120 & 24 & & & 1 & 24\\
[3pt]
\hline \hline
\\[-9pt]
& \mathrm{J}_1 &
\mathrm{O'N} &
\mathrm{J}_3 &
\mathrm{Ru} &
\mathrm{J}_4 &
\mathrm{Ly} \\
\hline
\H_2(G) & 1 & 3 & 3 & 2 & 1 & 1\\
\H_3(G) & 30 & 8 & 15 & ? & 1 & 1 \\
\H_3(n.G) & & 8 & 3\times 15 & ? & & \\
\hline\hline
\\[-9pt]
& \mathrm{He} & \mathrm{HN} &
\mathrm{Th} &
\mathrm{Fi}_{22} &
\mathrm{Fi}_{23} &
\mathrm{Fi}'_{24} &
\mathrm{B} &
\mathrm{M}
\\
\hline
\H_2(G) & 1 & 1 & 1 & 6 & 1 & 3 & 2 & 1\\
\H_3(G) & 12 & ? & ? & 1 & ? & ? & ? & 24 \times [{\leq} 4] \\
\H_3(n.G) & & & & 3 \times [{\leq} 4]& &? & ?\\
[3pt]
\hline \hline
\end{array}
\end{gather*}

An expression like
``$a \times b$''
 is short for
 $\bZ/a \oplus \bZ/b$.
 Question marks in the table denote groups for which we do not know the answer, and ``$[{\leq} 4]$'' denotes an unknown, possibly trivial, group of order dividing~$4$.
Further partial results for the groups $\mathrm{HN}$, $\mathrm{Th}$, $\Fi_{23}$, and $\Fi_{24}'$ are listed in Section~\ref{sec:monsters}.

Only some entries in the table are original. The Schur multiplier row (the first row in the table) was computed over many years, partly in service of the classification of finite simple groups, and is available in the ATLAS~\cite{ATLAS}.
With $\bF_2$-coefficients, the entire cohomology rings of many of the smaller sporadic groups are listed in \cite{MR2035696},
{and at large primes the cohomology rings of many sporadic groups are computed in \cite{CBThomasI,CBThomasII}}.
The Mathieu entries are reviewed in~\cite{SE09}. Significantly,
$\H_3(\rM_{24})$ was first computed in that paper using Graham Ellis's software package ``HAP'', which we have found can also determine $\H_3(G)$ for $G \in \{\mathrm{HS},2\mathrm{HS},\mathrm{J}_2, 2\mathrm{J}_2, \mathrm{J}_1,\mathrm{J}_3,\mathrm{McL}\}$ using the permutation models given in the ATLAS.
For the larger groups $G$, although HAP cannot calculate $\H_3(G)$ on its own, it played an essential role in our calculations, as did the ``Cohomolo'' package by Derek Holt.

\subsection{Motivation}\label{intro:motivation}
If $G$ is a compact simple Lie group, or a finite cover of a compact simple Lie group, the cohomology of its classifying space can be complicated at small primes but one always has $\H^4(BG;\bZ) \cong \bZ$; see~\cite{HenriquesWZW} for a proof and some discussion of its role in conformal field theory. In unpublished work~\cite{Grodal}, Jesper Grodal has shown that, with finitely many exceptions, $\H^4(G;\bZ) \cong \bZ/\big(q^2-1\big)$ whenever $G$ is a simple finite group which arises as the $\bF_q$-points of a~split and simply connected algebraic group over $\bF_q$. Part of our motivation has been to see whether we could discern any patterns in~$\H^4(G;\bZ)$ when~$G$ is sporadic.

We have also been inspired by the idea that $3$-cocycles $G \times G \times G \to \mathrm{U}(1)$ (when $G$ is finite, these represent classes in $\H^4(G;\bZ)$) can explain and predict some features of moonshine \cite{CLW, GPRV,MR3539377,MR2500561}. Such a cocycle can arise as the gauge anomaly of a~$G$-action on a conformal field theory. Even in the newer examples of moonshine where no conformal-field-theoretic explanation is known, there are some numerical hints about this cocycle. For example, Duncan--Mertens--Ono have used our calculations to explore a cocycle in their ``O'Nan moonshine'' \cite[Section~3]{DMO1}.

To some extent these hints can be pursued in an elementary way in pure group theory.
If $s$ and $t$ are a pair of commuting elements in a finite group $G$, we may define the following infinite group:
\begin{gather*}
\Gamma(s,t) := \left\{\left(
\left(\begin{matrix}
a & b \\ c & d
\end{matrix}
\right), g \right) \in \SL_2(\bZ) \times G\, \bigg| \,gsg^{-1} = s^a t^b \text{ and } gtg^{-1} = s^c t^d
\right\}.
\end{gather*}
It is the fundamental group of one of the components of the moduli stack of pairs $(E,T)$, where~$E$ is an elliptic curve and $T$ is a $G$-torsor over $E$.
If there is a natural family of McKay--Thompson series attached to~$G$, one expects that their modularity properties (and more ambiguously, their mock modularity properties) can be expressed in terms of a holomorphic line bundle on this space, or equivalently in terms of a $\Gamma(s,t)$-equivariant line bundle on the upper-half plane. The topological types of such line bundles are parametrized by the finite group $\H^2(\Gamma(s,t);\bZ)$, which is the target of a transgression map $\H^4(G;\bZ) \to \H^2(\Gamma(s,t);\bZ)$ \cite[Section~2]{MR2500561}.

\section{Preliminaries}

\subsection{{Notation}}\label{subsec:notation}

We will generally follow the ATLAS naming conventions for finite groups. We will write both~``$\bZ/n$'' and plain~``$n$'' for the cyclic group of order~$n$. When~$q$ is a prime power, we will occasionally use ``$q$'' to denote the finite field $\bF_q$ of that order. Physicists typically denote the cyclic group of order~$n$ by~$\bZ_n$. Following mathematics conventions, we will instead reserve~$\bZ_p$, where~$p$ is prime, for the ring of $p$-adic integers.

We will write ``$N.J$'' or ``$NJ$'' for an extension with normal subgroup $N$ and quotient $J$. Extensions that are known to split are written with a colon ``$N:J$'', and extensions which are known not to split are written with a raised dot ``$N \cdot J$''. The name ``$p^n$'', where $p$ is prime, denotes an elementary abelian group of that order, and if $n$ is even then ``$p^{1+n}$'' denotes an extraspecial group of that order. (There are two such extraspecial groups, called ``$p^{1+n}_\pm$''.)

We diverge from the ATLAS in the names for orthogonal groups.
The group called ``$\rO_n(q)$'' in the ATLAS is not the $n\times n$ orthogonal group over $\bF_q$. Rather, the ATLAS uses ``$\rO_n(q)$'' for the simple subquotient of the orthogonal group. To avoid confusion, we will follow Dieudonn\'e and write ``$\Omega_n(q)$'' for this simple group. We will care only about the case when $n \geq 5$ is odd~-- when $n$ is even, there are two orthogonal groups, called $\Omega^\pm_n(q)$. When $n\geq 5$ and $q$ is odd, $\Omega_n(q)$ is the commutator subgroup of $\mathrm{SO}_n(\bF_q) = \Omega_n(q):2$, and is the image of $\Spin_n(\bF_q) = 2.\Omega_n(q)$ in $\mathrm{SO}_n(\bF_q)$, and is the kernel of the ``spinor norm'' $\mathrm{SO}_n(\bF_q) \to \bF_q^\times / \{\text{squares}\} \cong \bZ/2$.

Conjugacy classes of order $n$ are named $n\ra$, $n\rb$, $n\rc$, and so on. For simple groups the conjugacy classes are ordered by size of the centralizer (from largest to smallest). In all cases we follow GAP's character table library, which includes a copy of the ATLAS character tables, for the names of conjugacy classes. The online version of the ATLAS \cite{ATLASonline} includes a number of irreducible modular representations. (We henceforth adopt the standard abbreviation ``irrep'' for ``irreducible representation''.) These are typically assigned letters ``$a$'', ``$b$'', etc., to distinguish irreps of the same dimension and characteristic.

If $G$ is a finite group, the names ``$\H_*(G)$'' and ``$\H^*(G)$'' always refer to group (co)homology, or equivalently the space (co)homology of the classifying space $BG$ of $G$. When $G$ is a Lie group, we will explicitly write $\H_*(BG)$ and $\H^*(BG)$ to avoid confusion with the (co)homology of the underlying manifold of $G$. Cohomology groups of $G$ with (twisted) coefficients in~$A$ are denoted~$\H^*(G;A)$. We sometimes abbreviate $\H^*(G;\bZ)$ by just $\H^*(G)$. All homology groups in this paper are with $\bZ$-coefficients.

\subsection{General methods}\label{sec:methods}

In this section and the next we review some standard techniques in group cohomology, which we return to repeatedly in the following sections. These techniques are by no means due to us~-- we employed them successfully in \cite{jft} to calculate the cohomology of Conway's largest sporadic group, and find in this paper that they also apply to most of the other sporadic groups. These techniques are designed to understand the cohomology groups of a finite group $G$ and not, say, to compute explicit resolutions of $\bZ$ over $\bZ[G]$.

The first {technique is to compute}
 the $p$-primary part of $\H^4(G;\bZ)$, which we denote by $\H^4(G;\bZ)_{(p)}$, one prime at a time. An upper bound for the $p$-primary part is provided by the following lemma \cite[Section~XII.8]{MR0077480}:

\begin{Lemma}\label{transfer restriction}
Let $G$ be a finite group and let $S \subseteq G$ be a subgroup that contains a Sylow $p$-subgroup for some prime $p$. The {restriction} map $\alpha \mapsto \alpha|_S\colon \H^k(G;\bZ)_{(p)} \to \H^k(S;\bZ)_{(p)}$ is an injection onto a direct summand.
\end{Lemma}

Lemma \ref{transfer restriction}, together with some basic properties (which we review in some detail in Section~\ref{subsec:eag}) of $\H^4(\bZ/p;\bZ)$ and $\H^4(\bZ/p \times \bZ/p;\bZ)$, allows us to dispose of many of the larger primes, at least for sporadic groups:

\begin{Lemma}\label{large primes}Let $p$ be a prime and let $G$ be a finite group with strictly fewer than $(p-1)/2$ conjugacy classes of order $p$. If the $p$-Sylow subgroup of $G$ is isomorphic to $\bZ/p$ or to $\bZ/p \times \bZ/p$, then the $p$-part of $\H^4(G;\bZ)$ vanishes.
\end{Lemma}

If $p \geq 5$ and $G$ is a sporadic simple group whose $p$-Sylow has order $p$ or $p^2$, then one sees by inspecting the tables of conjugacy classes that the criterion applies unless $p = 5$ and $G \in \{\mathrm{J}_2,\mathrm{Suz}\}$. We will see in Lemma~\ref{lem:suz 5} that the 5-part of $\H^4(\Suz;\bZ)$ vanishes as well.

\begin{proof}In any group with fewer than $(p-1)/2$ conjugacy classes of order $p$, the cyclic subgroups $C \subset G$ of that order have the following property: there is a generator $h \in C$ and an element $x \in G$ such that $xhx^{-1} = h^a$, where $a$ is neither $1$ nor $-1$ mod $p$. Conjugation by such an $x$ acts trivially on $\H^\bullet(G;\bZ)$ but nontrivially on $\H^4(C;\bZ)$~-- indeed it scales a nontrivial element $t \in \H^2(H;\bZ) \cong \bZ/p$ to $at$ and the cup-square of that nontrivial element $t^2 \in \H^4(H;\bZ) \cong \bZ/p$ to $a^2 t^2$.
It follows that the image of the restriction map $\H^4(G;\bZ) \to \H^4(C;\bZ)$ is zero, for every order $p$-subgroup $C \subset G$.

Let $H$ be a $p$-Sylow subgroup of $G$, and consider the subgroup $X \subset \H^4(H;\bZ)$ that vanishes on every order-$p$ subgroup $C \subset H$. The discussion above shows that the image of the restriction map $\H^4(G;\bZ) \to \H^4(H;\bZ)$ lies in $X$, and by Lemma \ref{transfer restriction}, this restriction map is an injection on the $p$-primary part of $\H^4(G;\bZ)$. When $p$ is odd and $H$ is an elementary abelian $p$-group of rank at most two, $\H^4(H;\bZ) \cong \Sym^2(H^*)$ (see Lemma \ref{elemab}), and so $\H^4(H;\bZ)$ is detected on cyclic subgroups, i.e., $X = 0$.
\end{proof}

In many cases not covered by Lemma~\ref{large primes}, there is a maximal subgroup $S \subseteq G$ that contains a~$p$-Sylow, and that has shape $S = E.J$ where $E$ is either an elementary abelian or an extraspecial $p$-group. (See \cite{MR3682590} for a survey of maximal subgroups of finite groups.)
Sometimes we know $\H^4(J;\bZ)$, either by induction or by computer. The Lyndon--Hochschild--Serre (LHS) spectral sequence {(detailed for example in \cite[Section~6.8]{MR1269324})}
\begin{gather*}
E_2^{ij} = \H^i\big(J;\H^j(E;\bZ)\big) \implies \H^{i+j}(S;\bZ)
\end{gather*}
gives an upper bound for $\H^4(S;\bZ)$, and therefore for $\H^4(G;\bZ)_{(p)}$, in terms of $\H^4(J;\bZ)$, which we assume is known by earlier computations, together with the cohomology groups with twisted coefficients
\begin{gather*}
\H^0\big(J;\H^4(E;\bZ)\big), \qquad \H^1\big(J;\H^3(E;\bZ)\big), \qquad \H^2\big(J;\H^2(E;\bZ)\big).
\end{gather*}
The contribution from $\H^3\big(J;\H^1(E;\bZ)\big)$ is zero, since $\H^1(E;\bZ)= 0$ for every finite~$E$.
We describe the groups $\H^j(E;\bZ)$ for $j = {2,3,4}$ as $\Aut(E)$-modules in Section~\ref{sec p-groups}. We used extensively Derek Holt's software package ``Cohomolo'' to determine the groups $\H^1(J;-)$ and $\H^2(J;-)$, but sometimes the following vanishing criterion can be employed instead:

\begin{Lemma}\label{lemma:central character} Suppose that the center $Z(J)$ has order prime to $p$ and acts on $\H^j(E;\bZ)$ through a nontrivial character $Z(J) \to \bF_p^\times$. Then $\H^i\big(J; \H^j(E;\bZ)\big) = 0$ for all $i$.
\end{Lemma}
\begin{proof}The statement is vacuous when $p=2$, and so we assume $p$ is odd for the remainder of the proof.
Let $\bZ_p[J]$ denote the group ring of $J$ with coefficients in the $p$-adic integers $\bZ_p$.
For $j > 0$, $\H^j(E;\bZ)$ is a finite $p$-group, so $\H^i\big(J;\H^j(E;\bZ)\big) \cong \Ext^i_{\bZ_p[J]}\big(\bZ_p,H^j(E,\bZ)\big)$ when $\bZ_p$ is given the trivial $J$-action.

Let $\chi$ be the composite of the character $Z(J) \to \bF_p^{\times}$ with the Teichmuller isomorphism
{$\bF_p^\times \cong \bZ_p^\times[\mathrm{tor}]$, where $\bZ_p^\times[\mathrm{tor}] \subseteq \bZ_p^\times$ denotes the torsion subgroup,}
and let
\begin{gather*}
e = \frac1{|Z(J)|}\sum_{z \in Z(J)} \chi(z)^{-1} z
\end{gather*}
be the corresponding central idempotent in $\bZ_p[J]$, so that $\bZ_p[J] = e\bZ_p[J] \times (1-e)\bZ_p[J]$ as rings. Since $\chi$ is nontrivial there is a projective resolution $P_{\bullet} \to \bZ_p$ of the trivial $J$-module with $P_m = (1-e)P_m$ for every $m$. It follows that $\Ext^i_{\bZ_p[J]}(\bZ_p,M) = 0$, for all $i$, whenever~$M$ is a~$\bZ_p[J]$-module with $M = eM$.
\end{proof}

{The LHS spectral sequence allows us a comparison between the cohomology of a group and of its Schur cover.
Let $G$ be a finite group with $n \subseteq \H_2(G)$ such that $\H^1(G; \bZ/n) = 0$.
Then the corresponding central extension $nG$ is unique up to isomorphism. Consider the LHS spectral sequence for this extension.
Since the extension is central, $G$ acts trivially on $n$ and so on $\H^j(n; \bZ)$, and so we have an isomorphism of bigraded rings
\begin{gather*} E_2^{ij} \cong \H^i(G; \H^j(n;\bZ)) \cong \H^i(G; \bZ[y]/(ny)), \end{gather*}
where $y$ has bidegree $(i,j) = (0,2)$; see, e.g., \cite[Section~II.8]{SerresThesis} and \cite[Section~II.5]{HochschildSerre}.
Using that $\H^1(G;\bZ/n) = 0$, in total degree $\leq 5$ the $E_2$ page reads:
\begin{gather*}
\begin{array}{cccccc}
0 \\
(\bZ/n)y^2 & 0 \\
0 & 0 & 0\\
(\bZ/n)y & 0 & \H^2(G;\bZ/n) & \H^2(G;\bZ/n) \\
0 & 0 & 0 & 0 & 0 \\
\bZ & 0 & \H^2(G;\bZ) & \H^3(G;\bZ) & \H^4(G;\bZ) & \H^5(G;\bZ)
\end{array}
\end{gather*}
It follows that the pullback $\H^4(G;\bZ) \to \H^4(nG;\bZ)$ is an injection. (Such pullbacks are examples of edge maps, described for example in \cite[Section~6.8.2]{MR1269324}.)}

Let us focus on the case when $n$ is a power of a prime $p$, and restrict to $p$-parts. Then $\H_1(G)_{(p)} = \H^2(G;\bZ)_{(p)} = 0$. If furthermore $\H_2(G)_{(p)}$ is cyclic, then $\H^2(G;\bZ/n) \cong \bZ/n$, and we have the $E_2$ page
\begin{gather*}
\begin{array}{cccccc}
0 \\
(\bZ/n)y^2 & 0 \\
0 & 0 & 0\\
(\bZ/n)y & 0 & \bZ/n & \H^3(G;\bZ/n) \\
0 & 0 & 0 & 0 & 0 \\
\bZ & 0 & 0 & \H^3(G)_{(p)} & \H^4(G)_{(p)} & \H^5(G)_{(p)}
\end{array}
\end{gather*}
The universal coefficient theorem describes $\H^3(G;\bZ/n)$ as an extension
\begin{gather*} \H^3(G;\bZ/n) = \bigl[ \H^3(G)_{(p)} \otimes (\bZ/n)y \bigr].\hom(\H_3(G), \bZ/n).\end{gather*}
The $d_2$ differential vanishes for degree reasons, and so $E_3^{ij} = E_2^{ij}$.

The extension $nG \to G$ splits when pulled back along itself, which implies that pullback $\H^3(G)_{(p)} \to \H^3(nG)_{(p)}$ has kernel of order $n$, forcing the differential $d_3\colon (\bZ/n)y \to \H^3(G)_{(p)}$ to be an inclusion.
The Leibniz rule then determines $d_3\big(y^2\big)$. If for instance $\H_2(G)_{(p)}$ is cyclic of order~$N$, then so is $\H^3(G)_{(p)}$; calling its generator ``$x$'', we have $d_3 y = (N/n)x$ and $d_3 y^2 = (2N/n)xy$, where $xy$ is the generator of the submodule $\bZ/n \cong \H^3(G)_{(p)} \otimes (\bZ/n)y \subset \H^2(G;\bZ/n)$.

All together we learn:

\begin{Lemma}\label{lemma:schur cover} Let $G$ be a finite group.

If $p$ is an odd prime such that $\H_1(G)_{(p)} = 0$ and $\H_2(G)_{(p)} = p$, then the pullback map $\H^4(G;\bZ) \to \H^4(pG;\bZ)$ is an injection with cokernel of order dividing $p$, and all classes in $\H^4(pG;\bZ)$ restrict trivially to the central $p \subseteq pG$.

If $\H_1(G)_{(2)} = 0$ and $\H_2(G)_{(2)}$ is $($nontrivial and$)$ cyclic, then the pullback $\H^4(G;\bZ) \to \H^4(2G;\bZ)$ is an injection with cokernel of order dividing $4$, and if the cokernel has order $4$ then there are classes in $\H^4(2G;\bZ)$ with nontrivial restriction to the central $2 \subset 2G$.

If $\H_1(G)_{(2)} = 0$ and $\H_2(G)_{(2)} = 4$, then the pullback $\H^4(G;\bZ) \to \H^4(4G;\bZ)$ is an injection with cokernel of order dividing~$8$; again equality forces there to exist a class in $\H^4(4G;\bZ)$ with nontrivial restriction to the central $4\subseteq 4G$, and all classes in $\H^4(4G;\bZ)$ vanish when restricted to the central $2 \subset 4 \subset 4G$.
\end{Lemma}

\begin{Lemma}\label{lemma:cokers} Let $p$ be an odd prime such that $\H_1(G)_{(p)} = 0$ and $\H_2(G)_{(p)} = p$. Let $pG$ denote a~nonsplit central extension of $G$ by the group $\bZ/p$. Suppose that a~$p$-Sylow $S \subseteq G$ also has $\H^1(S;\bZ/p) = 0$, and that the central extension $pG$, when restricted to $S$, is nonsplit. Then the pullback map $\H^4(pG;\bZ) \to \H^4(pS;\bZ)$ induces an injection
\begin{gather*}
\coker\big(\H^4(G;\bZ) \to \H^4(pG;\bZ)\big)\hookrightarrow \coker\big(\H^4(S;\bZ) \to \H^4(pS;\bZ)\big).
\end{gather*}
This injection is an isomorphism if $S$ contains the $p$-Sylow of $G$.
\end{Lemma}
\begin{proof} By Lemma~\ref{lemma:schur cover}, $\coker\big(\H^4(G;\bZ) \to \H^4(pG;\bZ)\big)$ and $\coker\big(\H^4(S;\bZ) \to \H^4(pS;\bZ)\big)$ are each {either trivial or of order~$p$}. We need only to show that if $\coker\big(\H^4(G;\bZ) \to \H^4(pG;\bZ)\big) = \bZ/p$, then $\coker\big(\H^4(S;\bZ) \to \H^4(pS;\bZ)\big) = \bZ/p$.

 Consider spectral sequence for the extension $pG \to G$ discussed before Lemma~\ref{lemma:schur cover}: we see that
 $\coker\big(\H^4(G;\bZ) \to \H^4(pG;\bZ)\big) = p$ if and only if the $d_3\colon E_3^{22} \to E_3^{50}$
 vanishes. Let $\alpha \in \H^2(G;\bZ/p) {{}\cong E_3^{22}}$ denote the generator classifying the extension $pG$. Then $d_3\colon \alpha \mapsto \operatorname{Bock}\big(\alpha^2\big)$, where $\operatorname{Bock} \colon \H^4(G;\bZ/p) \to \H^5(G;\bZ)$ denotes the integral Bockstein. This can be confirmed by comparing the spectral sequence for $\H^*(pG;\bZ)$ with the one for $\H^*(pG;\bZ/p)$.

But then $\operatorname{Bock}\big((\alpha|_S)^2\big) = \operatorname{Bock}\big(\alpha^2\big)|_S$ also vanishes, and so $\coker\big(\H^4(S;\bZ) \to \H^4(pS;\bZ)\big) \allowbreak = p$ by the spectral sequence for the extension~$pS$. Conversely, assuming $S$ contains the $p$-Sylow in~$G$, if $\operatorname{Bock}\big(\alpha^2\big)|_S = 0$, then $\operatorname{Bock}\big(\alpha^2\big) = 0$ by Lemma~\ref{transfer restriction}.
\end{proof}

As we have mentioned, each page of the LHS spectral sequence provides an upper bound for~$\H^4(G)_{(p)}$. We can improve this upper bound whenever we can show that the images of the two maps
\begin{gather*}
\H^4(J;\bZ) \to \H^4(S;\bZ) \leftarrow \H^4(G;\bZ)
\end{gather*}
have trivial intersection. We can often prove this by restricting generators of~$\H^4(J;\bZ)$ and $\H^4(G;\bZ)$ to cyclic subgroups and showing that no class in $\H^4(S;\bZ)$ can simultaneously enjoy the restrictions mandated by both $\H^4(J;\bZ)$ and $\H^4(G;\bZ)$. For these calculations, we rely on GAP's character table library, which includes a copy of the ATLAS and, provided it contains the subgroup~$S$, knows how conjugacy classes fuse along the maps~$S \to G$ and $S \to J$.

\subsection{Characteristic classes} \label{subsec:charclasses}

\looseness=-1 With the improved upper bound in hand, the last step is to give a lower bound for $\H^4(G;\bZ)$. In almost all cases these come from the characteristic class of a representation $V\colon G \to K$, where~$K$ is a Lie group. Usually we can take $K = \SU(N)$ or $\Spin(N)$, for which the characteristic classes are called, respectively, the second Chern class $c_2$ and the first fractional Pontryagin class $\frac{p_1}2$. In two cases these ``classical'' characteristic classes~$c_2$ and $\frac{p_1}2$ are not strong enough, and we appeal to the Lie groups $K = \mathrm{E}_6$ and $\mathrm{E}_8$. For some of the Monster sections, it is not possible for Lie-group-valued representations to give a strong enough lower bound, and we instead appeal to the construction of~\cite{JFmoonshine} to provide a ``monstrous characteristic class'' of a representation of~$G$ in~$\rM$.

\looseness=-1 We now review the story of $c_2$ and $\frac{p_1}2$. See also~\cite{MR878978} for a detailed treatment of characteristic classes of finite groups.
Suppose $N \geq 2$. Then $\H^4(B\mathrm{U}(N);\bZ) \cong \bZ^2$, with standard generators the square of the first Chern class $c_1^2$ and the second Chern class~$c_2$. The first of these restricts trivially along $\SU(N) \subset \mathrm{U}(N)$, and so vanishes when restricted to any finite simple group $G$; but if $V\colon G \to \mathrm{U}(N)$ is an $N$-dimensional representation, then $c_2(V) \in \H^4(G;\bZ)$ is a~potentially-interesting class. Similarly, provided $N \geq 5$, the generators of $\H^4(B\mathrm{SO}(N);\bZ) \cong \bZ$ and $\H^4(B\Spin(N);\bZ) \cong \bZ$ are called the first Pontryagin class $p_1$ and the first fractional Pontryagin class $\frac{p_1}2$. Like the symbol $\frac{p_1}{2}$ suggests, the pullback $\H^4(B\mathrm{SO}(N);\bZ) \to \H^4(B\Spin(N);\bZ)$ along the double cover sends $p_1$ to $2 \times \frac{p_1}2$.
There are also maps between $\SU(N)$ and $\mathrm{SO}(N)$ and $\Spin(2N)$ which either complexify a real representation or produce the underlying real representation of a complex representation. The characteristic classes restrict along these maps as
\begin{alignat*}{4}
 \SU(N) & \to \Spin(2N), \qquad & \mathrm{SO}(N) & \to \SU(N), &\\
 -c_2 & \mapsfrom \tfrac{p_1}2, \qquad & -p_1 & \mapsfrom c_2. &
\end{alignat*}
These classes are {\it stable} in the sense that they are preserved along the standard inclusions $\SU(N) \subseteq \SU(N+1)$ and $\Spin(N) \subseteq \Spin(N+1)$. When $N = 4$, $\H^4(B\Spin(4);\bZ)$ is not gene\-ra\-ted by $\frac{p_1}2$, but that class is still defined by restricting along the standard inclusion into $\Spin(N)$ for $N$ large.

To show that the Chern class $c_2(V)$ of an $N$-dimensional representation $V\colon G \to \mathrm{U}(N)$ has large order, it often suffices to restrict it to a cyclic subgroup $\langle g\rangle \subset G$. If $g$ has order~$n$, then $\H^4(\langle g\rangle;\bZ) \cong \bZ[t]/(nt)$, where the degree-$2$ generator $t$ is defined as the first Chern class~$c_1(\bC_1)$ of the one-dimensional representation $\bC_\zeta\colon g \mapsto \zeta = \exp(2\pi i /n) \in \mathrm{U}(1)$. The other 1-dimensional representations of $\langle g\rangle$ are its tensor powers $\bC_{\zeta^m} = \bC_\zeta^{\otimes m}\colon g \mapsto \zeta^m$, and $c_1(\bC_{\zeta^m}) = mt$ and $c_2(\bC_{\zeta^m}) = 0$. A~higher-degree representation splits over $\langle g\rangle$ as a sum of 1-dimensional representations. The Whitney sum formula says that for any group $G$ and representations~$V$,~$W$, we have
\begin{gather*} c_1(V \oplus W) = c_1(V) + c_1(W) \in \H^2(G;\bZ), \\ c_2(V\oplus W) = c_2(V) + c_2(W) + c_1(V)c_1(W) \in \H^4(G;\bZ).\end{gather*}
In particular, if $V|_{\langle g\rangle} = \bigoplus_k \bC_\zeta^{m_k}$, then
\begin{gather*} c_2(V)|_{\langle g\rangle} = \sum_{k<k'} m_k m_{k'} t^2 \in \H^4(\langle g\rangle;\bZ).\end{gather*}
The Chern classes are traditionally organized into a {\it total Chern class} of mixed degree $c(V) = 1 + \sum\limits_{i\geq 1}c_i(V) \in \H^\bullet(G;\bZ)$. The full Whitney sum formula then says that $c(V \oplus W) = c(V)c(W)$; for the one-dimensional representations of a cyclic group, $c(\bC_{\zeta^m}) = 1+mt$; and the above formula is the coefficient on $t^2$ of $c(V)|_{\langle g\rangle} = \prod_k (1 + m_k t)$.

A representation $V\colon G \to \SU(N)$ is called {\it real} if it factors, up to $\SU(N)$-conjugacy, through $\SO(N)$, i.e., if the representation preserves a nondegenerate symmetric bilinear form. For irreps, this occurs if and only if the Frobenius--Schur indicator of $V$ is $+1$. (A representation with indicator~$-1$ is called {\it quaternionic} and factors through a symplectic group.) Frobenius--Schur indicators are quick to compute from a character table for~$G$; they are listed in the ATLAS and easily accessed in GAP. A real representation $V\colon G \to \SO(N)$ is {\it Spin} if it factors (aka lifts) through $\Spin(N)$; a choice of factorization is also called a {\it spin structure}. This occurs if and only if the second Stiefel--Whitney class $w_2(V) \in \H^2(G; \bZ/2)$ vanishes. This happens automatically if $G$ is a~Schur cover of a simple group, as then $\H^2(G; \bZ/2) = 0$.

Given a real representation $V\colon G \to \SO(N)$ with complexification $V \otimes \bC\colon G \to \SU(N)$, the classes $p_1(V)$ and $c_2(V\otimes \bC)$ agree up to sign, and so the calculation can proceed as above. Calculating $\frac{p_1}2(V)$ for $V\colon G \to \Spin(N)$ can be harder. If $V$ factored through $W\colon G \to \SU(N/2)$, then the calculation would be easy, as then $\frac{p_1}2(V) = -c_2(W)$. In the cases of interest, this does not occur for the whole representation $V$ but does occur for its restriction $V|_{\langle g\rangle}$ to a cyclic subgroup.

The spin structure for a real representation $V\colon G \to \SO(N)$, if it exists, typically is not unique.
 Rather, the choices form a torsor for $\H^1(G;\bZ/2) = \hom(G;\bZ/2)$ (so in particular the lift is unique for quasisimple groups). Even though the lift is typically not unique, the class~$\frac{p_1}2(V)$, if it exists, depends only on the complex representation $V\colon G \to \mathrm{SU}(N)$ (since the factorization through $\SO(N)$ is unique):

\begin{Lemma}\label{lem:p12} Suppose $V_1,V_2\colon G \to \Spin(N)$ are two spin structures on the same real representation $V\colon G \to \SO(N)$. Then $\frac{p_1}2(V_1) = \frac{p_1}2(V_2) \in \H^4(G;\bZ)$.
\end{Lemma}

See \cite[Section~1.4]{jft} for an explanation of Lemma~\ref{lem:p12} in terms of the ``string obstruction''.

\begin{proof} The reason that spin structures form a torsor for $\H^1(G;\bZ/2)$ is the following. Let $c \in \Spin(N)$ denote the nontrivial element in $\ker(\Spin(N) \to \SO(N))$. There is a group homomorphism
 \begin{gather*} \alpha\colon \ \bZ/2 \times \Spin(N) \to \Spin(N), \qquad (i,g) \mapsto c^i g,\end{gather*}
 covering the standard projection $\bZ/2 \times \SO(N) \to \SO(N)$.
 Given $V_1$ and $V_2$ as above, there is a unique map $\phi\colon G \to \bZ/2$ such that
 \begin{gather*} V_2 = \alpha \circ (\phi,V_1).\end{gather*}

 Let $\pi\colon \bZ/2 \times \Spin(N) \to \Spin(N)$ denote the standard projection. Then $V_1 = \pi \circ (\phi,V_1)$. In particular, it suffices to show that the pullbacks of $\frac{p_1}2$ along the two maps $\alpha,\pi\colon \bZ/2 \times \Spin(N) \to \Spin(N)$ agree. But $\H^4\big(B(\bZ/2 \times \Spin(N))\big) = \H^4(\bZ/2) \oplus \H^4(B\Spin(N))$ by the K\"unneth formula, and
 \begin{gather*} \pi^*{\tfrac{p_1}2} = \bigl(0, {\tfrac{p_1}2}\bigr) \in \H^4(\bZ/2) \oplus \H^4(B\Spin(N)), \\
 \alpha^*{\tfrac{p_1}2} = \bigl({\tfrac{p_1}2}|_{\langle c\rangle}, {\tfrac{p_1}2}\bigr) \in \H^4(\bZ/2) \oplus \H^4(B\Spin(N)),\end{gather*}
 so it suffices to show that~$\frac{p_1}2$ has trivial restriction to $\bZ/2 \cong \langle c\rangle = \ker(\Spin(N) \to \SO(N))$.

 Suppose that $K$ is a compact connected Lie group with maximal torus $T \subseteq K$, and write $L = \hom(T, U(1)) = \H^2(BT)$ for its weight lattice.
 Then the restriction map $\H^4(BK) \to \H^4(BT) = \Sym^2(L)$ is an injection.
 When $K = \SO(N)$, there is a natural identification $L \cong \bZ^N$. Writing $e_1,\dots,e_N$ for the standard basis, we have $p_1 = \sum_i e_i^2$. The weight lattice $L'$ of $\Spin(N)$ is the extension of $L$ through the element $s = \frac12 \sum_i e_i$. Working in $\Sym^2 L'$, we have
 \begin{gather*} \tfrac{p_1}2 = 2s^2 - \sum_{i<j} e_i e_j.\end{gather*}
 But $e_i$ and $2s$ are in $L$ and so restrict trivially to $\langle c \rangle$, and so $\frac{p_1}2$ also restricts trivially.
\end{proof}

\section[Elementary abelian and extraspecial $p$-groups]{Elementary abelian and extraspecial $\boldsymbol{p}$-groups} \label{sec p-groups}

\subsection{Elementary abelian groups}\label{subsec:eag}

\begin{Lemma} \label{elemab}Let $E = p^n$ be an elementary abelian $p$-group and let $E^*:=\Hom(E,\mu_p)$, where~$\mu_p$ denotes the group of $p$th roots of unity in $\bC^*$.
\begin{enumerate}\itemsep=0pt
\item[$1.$] If $p = 2$, we have isomorphisms of $\GL(E)$-modules
\begin{gather*}
\H^2(E;\bZ) = E^*, \qquad \H^3(E;\bZ) = \Alt^2(E^*), \qquad \H^4(E;\bZ) = E^*.\Alt^2(E^*).\Alt^3(E^*),
\end{gather*}
where the last group on the right denotes a filtered $\GL(E)$-module whose subquotients are
 $E^*$, $\Alt^2(E^*)$, and $\Alt^3(E^*)$. The submodule $E^*.\Alt^2(E^*)$ is $\GL(E)$-isomorphic to $\Sym^2(E^*)$.
\item[$2.$] If $p$ is odd, we have isomorphisms of $\GL(E)$-modules
\begin{gather*}
\H^2(E;\bZ) = E^*, \qquad \H^3(E;\bZ) = \Alt^2(E^*), \qquad \H^4(E;\bZ) = \Sym^2(E^*) \oplus \Alt^3(E^*).
\end{gather*}
\end{enumerate}
\end{Lemma}

\begin{proof}See \cite[Proposition 2.2]{MR2320456} or \cite[Lemma 4.4]{jft}.
\end{proof}

If $V$ is an elementary abelian $p$-group, we regard it as an $\bF_p$-vector space in the obvious way. We may identify $E^*$ with the usual dual $\bF_p$-vector space to $E$ by fixing at the outset an isomorphism $\mu_p \cong \bZ/p$. We use $\Sym^n(V)$ and $\Alt^n(V)$ for the symmetric and exterior powers of~$V$; recall in positive characteristic these are defined as quotients of $V^{\otimes n}$ in the following way:
\begin{itemize}\itemsep=0pt
\item $\Sym^n(V):= \H_0\big(S_n;V^{\otimes n}\big)$ are the coinvariants of $V^{\otimes n}$ by the symmetric group action
\item $\Alt^n(V)$ is the quotient of $V^{\otimes n}$ by the subspace spanned by tensors with a repeated tensorand (tensors $v_1 \otimes \cdots \otimes v_n$ with $v_i = v_j$ for some $i \neq j$).
\end{itemize}
Though $\Sym^n(E^*)$ and $\Sym^n(E)^*$ are not isomorphic as $\GL(E)$-modules if $p \leq n$ (instead the dual of $\Sym^n(E^*)$ is the space of divided powers of $E$), let us record:

\begin{Lemma}
If $p$ is a prime and $E$ is an $\bF_p$-vector space, there is an isomorphism
\begin{gather*}
\Alt^n(E^*) \cong \Alt^n(E)^*
\end{gather*}
of $\GL(E)$-modules.
\end{Lemma}

\begin{proof}The pairing $V^{\otimes n} \otimes (V^*)^{\otimes n} \to \bZ/p$ given by
\begin{gather*}
\langle v_1 \otimes \cdots \otimes v_n, w_1 \otimes \cdots \otimes w_n\> = \sum_{\sigma \in S_n} (-1)^\sigma \langle v_1, w_{\sigma(1)}\rangle \cdots \langle v_n, w_{\sigma (n)}\rangle,
\end{gather*}
 where $(-1)^\sigma$ denotes the sign of the permutation $\sigma$, is $\GL(V)$-equivariant and descends to a~perfect pairing between $\Alt^n(V)$ and $\Alt^n(V^*)$.
\end{proof}

\subsection[Extraspecial $p$-groups for $p$ odd]{Extraspecial $\boldsymbol{p}$-groups for $\boldsymbol{p}$ odd}

If $p$ is prime, $E = p^n$ is an elementary abelian $p$-group and $\omega$ is a function $E \times E \to \bZ/p$, we define a multiplication on the set of formal monomials of the form $z^i t^u$ (where $i \in \bZ/p$ and $u \in E$) by the formula
\begin{gather*}
\big(z^i t^u\big)\big(z^j t^v\big) := z^{i+j + \omega(u,v)} t^{u+v}.
\end{gather*}
If $\omega$ is bilinear, this multiplication is associative, $z^0 t^0$ is a two-sided unit, and $z^{-i+\omega(u,u)}t^{-u}$ is the two-sided inverse to $z^i t^u$: we defined a group that we denote by $(p.E)_{\omega}$. The groups associated to $(p.E)_{\omega}$ and $(p.E)_{\omega'}$ are isomorphic if $\omega - \omega'$ can be written as $j(u+v) - j(u) - j(v)$ for some function $j\colon E \to \bZ/p$~-- in particular if $p$ is odd then $\omega(u,v)$ and
\begin{gather*}
\tfrac{1}{2} (\omega(u,v) - \omega(v,u) ) = \omega(u,v) - \tfrac{1}{2} (\omega(u+v,u+v) - \omega(u,u)-\omega(v,v) )
\end{gather*}
determine isomorphic groups, so when~$p$ is odd we may as well assume that $\omega \in \Alt^2(E^*)$ is skew-symmetric. The center contains $z$, and if $p$ is odd it is generated by~$z$ if and only if~$\omega$ is nondegenerate; in that case $n = 2m$ and $(p.E)_{\omega} = p^{1+2m}$ is a copy of the extraspecial $p$-group of exponent $p$. (The extraspecial group of exponent~$p^2$ comes from a non-bilinear cocycle $\omega\colon E \times E \to \bF_p$. The extraspecial groups of order $2^{1+2m}$ will be treated in Section~\ref{sub.extraspecial 2group}; the group $(p.E)_{\omega}$ that we have defined is always elementary abelian when $p = 2$).

The automorphism group of $p^{1+2m}$ is $E:\mathrm{GSp}(E,\omega)$, where $E$ acts by inner automorphisms $z^i t^u \mapsto z^{i+2\omega(v,u)} t^u$ and{\samepage
\begin{gather*}
\mathrm{GSp}_{2m}(E,\omega) = \{(g,a) \,|\, g\colon E \to E,\, a \in \GL_1(\bF_p),\, \omega(gu,gv) = a\omega(u,v)\}
\end{gather*}
acts by $(g,a) (z^i t^u) = z^{ai} t^{gu}$. The scalar $a = a(g)$ is determined by $g$.}

Let $L_\omega \subseteq \Alt^2(E^*)$ denote the line spanned by $\omega$. It is a one-dimensional $\mathrm{GSp}$-submodule by construction, and we write $L_{\omega}^n$ for its $n$th tensor power. Note $L_{\omega}^{-1} = L_{\omega}^*$. If $\omega$ is nondegenerate then $E \otimes L_\omega \cong E^*$ as $\mathrm{GSp}$-modules, via the map which sends $u \otimes \omega$ to the functional $\omega(u,-)$. Provided $p$ is odd, we have a splitting
\begin{gather*}
\Alt^2(E^*) = L_{\omega} \oplus \Alt^2(E^*)_{\omega},
\end{gather*}
where $\Alt^2(E^*)_{\omega}$ is the kernel of the projection $\Alt^2(E^*) \cong \Alt^2(E \otimes L_{\omega}) \cong \Alt^2(E^*)^* \otimes L_{\omega}^{2} \to L_{\omega}^{-1} \otimes L_{\omega}^2$ dual to the inclusion $L_{\omega} \to \Alt^2(E^*)$.

If $m \geq 2$ we also have an inclusion $E^* \otimes L_{\omega} \to \Alt^3(E^*)$ sending $f \in E^*$ to $f \wedge \omega$.

\begin{Lemma}\label{lem:extraspecial-odd}
Let $p$ be an odd prime, let $E = p^{2m}$ be an elementary $p$-group and let $\omega \in \Alt^2(E^*)$ be a nondegenerate symplectic form. Then if $m \geq 2$,
\begin{gather*}
\H^2\big(p^{1+2m};\bZ\big) \cong E^*, \qquad \H^3\big(p^{1+2m};\bZ\big) \cong \Alt^2(E^*)_\omega,
\end{gather*}
as $\mathrm{GSp}_{2m}$-modules. If $m \geq 3$,
\begin{gather*}
\H^4\big(p^{1+2m};\bZ\big) \cong \Sym^2(E^*) \oplus \Alt^3(E^*)/(E^* \otimes L_{\omega}) ,
\end{gather*}
while if $m = 2$,
\begin{gather*}
\H^4\big(p^{1+4};\bZ\big) \cong \Sym^2(E^*).\big(\Alt^2(E^*)_{\omega} \otimes L_{\omega}\big),
\end{gather*}
a possibly nontrivial extension of $\Alt^2(E^*)_{\omega}$ by $\Sym^2(E^*)$.
\end{Lemma}

\begin{proof}We consider the action of $\GSp$ on the LHS spectral sequence
\begin{gather*}
\H^s(E;\H^t(p)) \Rightarrow \H^{s+t}(p.E).
\end{gather*}
We have $\H^2(p) = L_{\omega}$ and $\H^4(p) = L_{\omega}^{2}$ in the left $s = 0$ column. The bottom $t = 0$ row is computed in Lemma~\ref{elemab}.
To compute the $t=2$ row, recall that, provided $p$ is odd, $\H^\bullet(E;\bF_p)$ is the graded-commutative $\bF_p$-algebra generated by a copy of $E^*$ in degree $1$ and a second copy of $E^*$ in degree $2$; in particular:
\begin{gather*}
 \H^1(E;\bF_p) \cong E^*, \qquad \H^2(E;\bF_p) \cong \Alt^2(E^*) \oplus E^*, \\ \H^3(E;\bF_p) \cong \Alt^3(E^*) \oplus (E^* \otimes E^*) \cong \Alt^3(E^*) \oplus \Alt^2(E^*) \oplus \Sym^2(E^*). \end{gather*}

All together, we have on the $E_2$-page:
\begin{gather*}
\begin{array}{ccccc}
L_{\omega}^{2} \\
0 & 0 & 0 \\
L_{\omega} & E^* \otimes L_{\omega} & (\Alt^2(E^*) \oplus E^*) \otimes L_{\omega} & \Alt^2(E) \otimes L_\omega \oplus \cdots \\
0 & 0 & 0 & 0 & 0 \\
\bZ & 0 & E^* & \Alt^2(E^*) & \Sym^2(E^*) \oplus \Alt^3(E^*)
\end{array}
\end{gather*}

The $d_2$ differential vanishes and the $d_3$ differentials $L_\omega \to \Alt^2(E^*)$, $E^* \otimes L_{\omega} \to \Alt^3(E^*)$, and $L_{\omega}^{2} \to \Alt^2(E^*) \otimes L_\omega$ are the injections discussed above. Indeed, the LHS spectral sequence is constructed so that $d_3$ sends the generator $\omega \in L_\omega$ to the extension class $\omega \in \Alt^2(E^*)$, and so it sends $\omega^2 \in L_\omega^2$ to $2\omega \,d_3\omega$. The claim for $E^* \otimes L_{\omega} \to \Alt^3(E^*)$ follows from comparing with the $\bF_p$-cohomology.

It remains to understand $d_3\colon \big(\Alt^2(E^*) \oplus E^*\big) \otimes L_{\omega} \to \H^5(E)$.
We claim that this map is an injection when $m\geq 3$, and that when $m=2$ its kernel is $\Alt^2(E^*)_\omega \otimes L_{\omega} \subseteq \Alt^2(E^*)$.
Note also that when $m=2$, the map $E^* \otimes L_{\omega} \to \Alt^3(E^*)$ is an isomorphism.
In this range of degrees, the sequence stabilizes after page $4$,
 and so on the $E_\infty$ page we see
\begin{gather*}
 \begin{array}{ccccc}
 0 \\
 0 & 0 & 0 \\
 0 & 0 & \Alt^2(E^*)_\omega \otimes L_{\omega} \\
 0 & 0 & 0 & 0 & 0 \\
 \bZ & 0 & E^* & \Alt^2(E^*)_\omega & \Sym^2(E^*) \\
 \end{array}
\end{gather*} if $m=2$ and
\begin{gather*}
 \begin{array}{ccccc}
 0 \\
 0 & 0 & 0 \\
 0 & 0 & 0 \\
 0 & 0 & 0 & 0 & 0 \\
 \bZ & 0 & E^* & \Alt^2(E^*)_\omega & \Sym^2(E^*) \oplus \Alt^3(E^*)/(E^* \otimes L_\omega)
 \end{array}
\end{gather*}
if $m\geq 3$.\end{proof}

\subsection{Extraspecial 2-groups} \label{sub.extraspecial 2group}

If $E$ is an elementary abelian $2$-group then any central extension $2.E$ is determined up to isomorphism by the function
\begin{gather*}
Q\colon \ E \to \bF_2, \qquad Q(v) = \begin{cases} 1 & \text{if the lifts of $v$ in $2.E$ have order $4$,}
\\
0 & \text{otherwise,}
\end{cases}
\end{gather*}
which is a quadratic form. It is not usually possible to write the multiplication explicitly in terms of $Q$~-- indeed if~$Q$ is nondegenerate and $E$ has rank~$6$ or more the orthogonal group of~$Q$ (which we denote by~$O(Q)$) does not act on $2.E$~\cite{MR0476878}. But $O(Q)$ still acts on the cohomology of~$2.E$.

The LHS spectral sequence begins:
\begin{gather*}
\begin{array}{ccccc}
2 \\
0 & 0 & 0\\
2 & E^* & \Sym^2(E^*) \\
0 & 0 & 0 & 0 & 0 \\
\bZ & 0 & E^* & \Alt^2(E^*) & E^*.\Alt^2(E^*).\Alt^3(E^*)
\end{array}
\end{gather*}
We first wish to describe the $d_3$ differential. To do so, recall first that $\H^\bullet(E)$ injects into $\H^\bullet(E;\bF_2) \cong \Sym^\bullet(E^*)$ as the subalgebra in the kernel of the derivation $\Sq^1\colon \Sym^\bullet(E^*) \to \Sym^{\bullet+1}(E^*)$.
Identifying $\H^\bullet(E)$ with its image in $\H^\bullet(E;\bF_2)$,
 the $d_3$ differential sends $f \in E_2^{i2} \cong \Sym^i(E^*)$ to $\Sq^1(fQ) \in \Sym^{i+3}(E^*)$. In particular, it sends the generator of the $2$ in degree $(0,2)$ to $\Sq^1(Q) \in \Sym^3(E^*)$. The image of $\Sq^1\colon \Sym^2(E^*)$ to~$\Sym^3(E^*)$ is isomorphic to $\Alt^2(E^*)$, and under this isomorphism $\Sq^1$ takes $Q$ to its underlying alternating form $B_Q(x,y) = Q(x+y) - Q(x) - Q(y)$.

Let us suppose that $Q$ is nondegenerate and $E = 2^{2m}$. Then in particular $B_Q \neq 0$, so that $d_3\colon 2 \to \Alt^2(E^*)$ is an injection. Let $f \in E^*$ in degree $(1,2)$ and consider $d_3(f) = \Sq^1(fQ) = f^2Q + f\Sq^1(Q)$. Since $\Sym^\bullet(E^*)$ has no zero-divisors, if $f\neq 0$ but $d_3(f)=0$, then we must have $\Sq^1(Q) = fQ$. This cannot happen when $m\geq 2$, and so $d_3\colon E^* \to E^*.\Alt^2(E^*).\Alt^3(E^*)$ is an injection in this case.
(When $m=1$, it is an injection when $Q$ has Arf invariant $-1$ and is not an injection when $Q$ has Art invariant $+1$.)
Thus, provided $m\geq 2$, we find
\begin{gather*} \H^1(2.E) \cong E^*, \qquad \H^2(2.E) = \Alt^2(E^*)/B_Q. \end{gather*}

The $d_3$ differential emitted by the $\Sym^2(E^*)$ in degree $(2,2)$ always has kernel~-- $Q$ itself~-- and nothing more provided $m\geq 2$. Finally, if $m \geq 3$, then $d_5\colon E_5^{04} \to E_5^{50}$ is nonzero, and the~$E_\infty$ page looks like
\begin{gather*}
\begin{array}{ccccc}
0 \\
0 & 0 \\
0 & 0 & Q \\
0 & 0 & 0 & 0 \\
\bZ & 0 & E^* & \Alt^2(E^*)/B_Q & X
\end{array}
\end{gather*}
with
\begin{gather*} X \cong \big(E^*.\Alt^2(E^*).\Alt^3(E^*)\big)/E^*. \end{gather*}
This can be simplified slightly. The inclusion $E^* \to E^*.\Alt^2(E^*).\Alt^3(E^*)$, sending $f \mapsto \Sq^1(fQ)$, does not land within the $E^*.\Alt^2(E^*) \cong \Sym^2(E^*)$ submodule, and so the composition $E^* \to E^*.\Alt^2(E^*).\Alt^3(E^*) \to \Alt^3(E^*)$ is nonzero. But $E^*$ is simple as an $O(Q)$-module, and so this map $E^* \to \Alt^3(E^*)$ is an injection. (It sends $f \mapsto f\wedge B_Q$.) Thus we can write
\begin{gather*} X \cong E^* .\Alt^2(E^*).\big(\Alt^3(E^*)/E^*\big). \end{gather*}

All together, provided $m\geq 3$,
\begin{gather*} \H^4(2.E) \cong \big(E^*.\Alt^2(E^*).\Alt^3(E^*)/E^*\big).2. \end{gather*}
The group $X$ is elementary abelian, although the extensions written above do not split $O(Q)$-equivariantly. The group $\H^4(2.E)$ is not elementary abelian; it is isomorphic to $(\bZ/2)^n \times (\bZ/4)$ for $n = \dim(X)-1 = \binom m 2 + \binom m 3 - 1$ when $m\geq 3$.

Finally, when $m=2$, whether $d_5\colon E_5^{04} \to E_5^{50}$ vanishes or not depends on the Arf invariant of $Q$. Indeed,
\begin{gather*} \H^4\big(2^{1+4}_+\big) = X.4 \cong 2^9\times 8, \qquad \H^4\big(2^{1+4}_-\big) = X.2 \cong 2^9\times 4. \end{gather*}
(Both cases are extensions of $X = E^*.\Alt^2(E^*).\Alt^3(E^*)/E^* \cong 2^{10}$.)

\section[Dempwolff groups, Chevalley groups and their exotic Schur covers]{Dempwolff groups, Chevalley groups\\ and their exotic Schur covers} \label{sec:Dempwolff}

\subsection{Dempwolff and Alperin groups}

In \cite{MR0357639}, Dempwolff determined that there were no nontrivial extensions of $\GL_n(\bF_2)$ by its defining representation on $2^n$, unless $n \leq 5$. Conversely, nontrivial extensions exist for $n = 3,4,5$; up to isomorphism there is a unique group which can serve as the extension, which we will call
\begin{gather*}
2^3 \cdot \GL_3(\bF_2), \qquad 2^4 \cdot \GL_4(\bF_2), \qquad 2^5 \cdot \GL_5(\bF_2).
\end{gather*}
The largest of these is studied in~\cite{MR0393276}, {though not proved to exist until~\cite{MR0409630,MR0399193}}. A similar group is the nonsplit Alperin-type group
\begin{gather*}
4^3 \cdot \GL_3(\bF_2).
\end{gather*}

\begin{Lemma}\label{lemma:Dempwolff} If $n = 3,4,5$, then $\H_3(\GL_n(\bF_2)) = \bZ/12$. Furthermore,
\begin{enumerate}\itemsep=0pt
\item[$1)$] $\H_3\big(2^3 \cdot \GL_3(\bF_2)\big) \cong \bZ/2 \oplus \bZ/8 \oplus \bZ/3$; 
\item[$2)$] $\H_3\big(2^4 \cdot \GL_4(\bF_2)\big) \cong \bZ/2 \oplus \bZ/4 \oplus \bZ/3$; 
\item[$3)$] $\H_3\big(2^5 \cdot \GL_5(\bF_2)\big) \cong \bZ/8 \oplus \bZ/3$; 
\item[$4)$] $\H_3\big(4^3 \cdot \GL_3(\bF_2)\big) \cong (\bZ/2)^2 \oplus \bZ/8 \oplus \bZ/3$. 
\end{enumerate}
\end{Lemma}

\begin{proof}HAP can handle all of these groups except the largest $2^5 \cdot \GL_5(\bF_2)$. (In Derek Holt's library of perfect groups, available in GAP, $2^3 \cdot \GL_3(\bF_2)$ is $\mathrm{PerfectGroup}(1344,2)$, $2^4 \cdot \GL_4(\bF_2)$ is $\mathrm{PerfectGroup}(322560,5)$, and $4^3 \cdot \GL_3(\bF_2)$ is $\mathrm{PerfectGroup}(10752,4)$. One may call these groups by number, have GAP find faithful permutation representations for them, and then feed those permutation groups to HAP~-- no further human involvement is needed.)

We will obtain $\H_3\big(2^5 \cdot \GL_5(\bF_2)\big) \cong \H^4\big(2^5 \cdot \GL_5(\bF_2)\big)$ from the LHS spectral sequence. Using the description from Lemma~\ref{elemab} of the $\GL_5(\bF_2)$-module structure on $\H^{\leq 4}(2^5)$, together with Cohomolo, we find the $E_2$ page of that spectral sequence is
\begin{gather*}
\begin{array}{ccccc}
 0 \\
 0 & 0 & 0 \\
 0 & 0 & \bZ/2 \\
 0 & 0 & 0 & 0 & 0 \\
 \bZ & 0 & 0 & 0 & \bZ/12
\end{array}
\end{gather*}
The $E_2^{22}$ entry here is the Dempwolff--Thompson--Smith computation $\H^2\big(\GL_5;\big(2^5\big)^*\big) = \bZ/2$, and confirmed by Cohomolo.

To complete the proof of (3), it suffices to give an element of $\H^4\big(2^5 \cdot \GL_5(2)\big)$ whose order is divisible by $8$.
There is a famous embedding, due to~\cite{MR0407149}, of $2^5 \cdot \GL_5(2)$ into the compact Lie group $\rE_8$.
Let us write~$e$ for the generator of $\H^4(B\rE_8)$. We will prove that the restriction $e|_{2^5 \cdot \GL_5(2)}$ is such an element.

For the remainder of the proof, let $V$ denote the $248$-dimensional adjoint representation of~$\rE_8$. The dual Coxeter number of $\rE_8$ is $h^\vee = 30$.
For any simple simply connected Lie group~$G$, the dual Coxeter number measures the ratio of the fractional Pontryagin class of the adjoint representation of $G$ with the generator of $\H^4(BG)$: \begin{gather*}\tfrac{p_1}2(\text{adj}) = h^\vee \in \bZ \cong \H^4(BG).\end{gather*}
In particular, $c_2(V) = - 60 e$. Since $60$ is divisible by $4$, to show that the order of $e|_{2^5 \cdot \GL_5(2)}$ is divisible by $8$, it suffices to show that the order $c_2(V)|_{2^5 \cdot \GL_5(2)}$ is divisible by $2$.

We will do so by finding a binary dihedral group $2D_8 \subseteq 2^5 \cdot \GL_5(2)$ such that $c_2(V)|_{2D_8}$ is nonzero. To find such a group, we look inside the normalizer of an order-$8$ element.
There are three conjugacy classes of elements of order $8$ in $2^5\cdot \GL_5(2)$.
The normalizer of class $8\rc$ is $\operatorname{SmallGroup}(64,151)$ in the GAP library. It can be built directly in GAP: the ATLASRep package includes a copy of $2^5\cdot \GL_5(2)$ as a permutation group on 7440 points; GAP can compute orders of centralizers and normalizers, and so in particular can identify class $8\rc$; then GAP can build the normalizer of an element of conjugacy class $8\rc$ as a subgroup of $2^5\cdot \GL_5(2)$. There are four conjugacy classes of order-$8$ elements in $\operatorname{SmallGroup}(64,151)$, and GAP checks that all four merge in $2^5\cdot \GL_5(2)$ to conjugacy class $8\rc$.

Finally, $\operatorname{SmallGroup}(64,151)$ contains a copy of the binary dihedral group $2D_8$ of order $16$. Since $2D_8$ is a finite subgroup of $\SU(2)$, its cohomology is easy to compute: in particular, $\H^4(2D_8)$ is cyclic of order $|2D_8| = 16$ and is generated by $c_2$ of the ``defining'' two-dimensional representation. As in \cite[Section 6]{jft}, let us index the irreducible representations:
\begin{gather*}
\begin{tikzpicture}
 \path (0,0) node (SW) {$V_1$}
 (1,1) node (A) {$V_6$}
 (0,2) node (NW) {$V_0$}
 (2,1) node (M) {$V_4$}
 (3,1) node (B) {$V_5$}
 (4,2) node (NE) {$V_2$}
 (4,0) node (SE) {$V_3$};
 \draw (A) -- (SW); \draw (A) -- (NW); \draw (A) -- (M);
 \draw (B) -- (SE); \draw (B) -- (NE); \draw (B) -- (M);
\end{tikzpicture}
\end{gather*}
In particular, $V_0$ is the trivial representation, $V_6$ is the ``defining'' two-dimensional irrep, $V_5$ is the other faithful irrep, $V_4$ is the two-dimensional real irrep of $D_8$, and $V_1$, $V_2$, and $V_3$ are the nontrivial one-dimensional irreps.

Character table constraints provide a unique fusion map $2D_8 \to 2^5\cdot \GL_5(2)$ sending the elements of order $8$ to conjugacy class $8\rc$. Along this map, the $248$-dimensional irrep $V$ of $2^5\cdot \GL_5(2)$ decomposes as
\begin{gather*} V|_{2D_8} = 15 V_0 \oplus 15 V_1 \oplus 15 V_2 \oplus 15 V_3 \oplus 30 V_4 \oplus 32 V_5 \oplus 32 V_6. \end{gather*}
Lemma 6.1 of \cite{jft} gives a formula for the second Chern class of any representation of $2D_8$ in which the representations
$V_2$ and $V_3$
 appear with the same coefficient. That formula is
\begin{gather*} c_2\left(\bigoplus n_i V_i\right) = 4n_4 + 9n_5 + n_6 \pmod {16}, \qquad \text{ if } n_1 = n_2,\end{gather*}
where we have identified $\H^4(2D_8) = \bZ/16$ by identifying $1 \in \bZ/16$ with $c_2(V_6)$. Applying this formula to the $248$-dimensional representation $V$ gives
\begin{gather*} c_2(V)|_{2D_8} = 8 \pmod {16}.\end{gather*}
In particular, $c_2(V)$ is nonzero in $\H^4(2D_8)$. As explained above, this implies that $\H^4\big(2^5\cdot \GL_5(2)\big)$ contains an element of order divisible by $8$ (namely, the restriction of the generator of $\H^4(B\rE_8)$), and so must be isomorphic to $\bZ/24$.
\end{proof}

\subsection{A few exotic Chevalley groups}

For the most part, any central extension of a finite Chevalley group $G(\bF_q)$ is the group of $\bF_q$-points of a central extension of the algebraic group $G$.
 In particular if $G$ is of simply connected type then the multiplier $\H_2(G(\bF_q))$ is usually zero. The finitely many exceptions were classified by Steinberg and Griess. Many of these exotic central extensions occur as centralizers in the sporadic groups.

 \begin{Lemma} \label{lemma:Sp6}$\H_3(\Sp_6(\bF_2)) = \bZ/2 \oplus \bZ/4 \oplus \bZ/3$ and $\H_3(2 \cdot \Sp_6(\bF_2)) = \bZ/2 \oplus \bZ/8 \oplus \bZ/3$.
\end{Lemma}

\begin{proof}We computed these using HAP. The computation of $\H_3(\Sp_6(\bF_2))$ is fast, but computing $\H_3(2 \cdot \Sp_6(\bF_2))$ took many hours. Two of the faithful permutation representations of $2 \cdot \Sp_6(\bF_2)$ have degrees $240$ and $276$ (the latter coming from the embedding $2 \cdot \Sp_6(\bF_2) \subseteq \Co_3$). Our laptop computer ran out of memory running HAP on the degree 240 model, and gave the above output after six hours for the degree 276 model.
\end{proof}

\begin{Lemma}\label{lemma:g2q}We have
\begin{gather*}
\H_3(\mathrm{G}_2(2)) = \bZ/2 \oplus \bZ/8 \oplus \bZ/3, \!\qquad \H_3(\mathrm{G}_2(3)) = \bZ/8 \oplus \bZ/3, \!\qquad \H_3(\mathrm{G}_2(5)) = \bZ/8 \oplus \bZ/3.\!
\end{gather*}
\end{Lemma}
Jesper Grodal has shown that $\H^4(\mathrm{G}_2(\bF_q))$ is cyclic of order $q^2 -1$ if $q = p^r$ with either $p$ or $r$ sufficiently large~\cite{Grodal}. The computations in the lemma show that this holds also for $q = 5$, but not $q = 3$ or $q = 2$.

\begin{proof}\looseness=-1 We computed $\mathrm{G}_2(2)$ and $\mathrm{G}_2(3)$ with HAP.
The order of $\rG_2(5)$ is $2^6.3^3.5^6.7.31$. The proof of Lemma~\ref{large primes} applies to this group~-- for $p=7$ and $31$, there are strictly fewer than
 $(p-1)/2$ conjugacy classes of order $p$~-- and so we must compute $\H^4(\rG_2(5))_{(p)}$ for $p=2$, $3$, and~$5$.

The $2$-Sylow in $\rG_2(5)$ is contained in the nonsplit extension $2^3\cdot \GL_3(2)$ whose cohomology, per Lemma~\ref{lemma:Dempwolff}(1), is $\H^4\big(2^3\cdot \GL_3(2)\big)_{(2)} = 2 \times 8$. According to \cite{MR1603199}, for $q = 1 \pmod {4}$,
\begin{gather*}
 \H^1(\rG_2(q);\bF_2) \cong \H^2(\rG_2(q);\bF_2) \cong 0 ,\qquad \H^3(\rG_2(q);\bF_2) \cong \H^4(\rG_2(q);\bF_2) \cong \bF_2, \\
 \Sq^1 = 0\colon \ \H^3(\rG_2(q);\bF_2) \to \H^4(\rG_2(q);\bF_2).
\end{gather*}
It follows that $\H^4(\rG_2(q))_{(2)}$ is cyclic of order at least $4$.
Setting $q = 5$ and recalling Lemma~\ref{transfer restriction}, we therefore find that $\H^4(\rG_2(5))_{(2)}$ is a cyclic direct summand of $\H^4\big(2^3\cdot \GL_3(2)\big)_{(2)} = 2 \times 8$ of order at least $4$, and so $\H^4(\rG_2(5))_{(2)} = 8$.

The $3$-Sylow in $\rG_2(5)$ is contained in a maximal subgroup of shape $\mathrm{U}_3(3):2$. HAP computes $\H_3(\mathrm{U}_3(3):2) = 2 \times 8 \times 3$. Conjugacy class $3\rb \in \rG_2(5)$ acts on the $124$-dimensional irrep with trace $1$, and so $c_2(\text{124-dim rep})|_{\langle 3\rb\rangle} \neq 0$. It follows that $\H^4(\rG_2(5))_{(3)} = 3$.

The $5$-Sylow in $\rG_2(5)$ is contained in a maximal subgroup of shape $5^{1+4}:\mathrm{GL}_2(5)$. The central $4 \subseteq \mathrm{GL}_2(5)$ acts on all of $5^4$ with the same faithful central character. It therefore acts with nontrivial central characters on $\H^j(5^{1+4})$ for $j \in \{1,2,3,4\}$, and so $\H^i\big(\mathrm{GL}_2(5), \H^j\big(5^{1+4}\big)\big) = 0$ for these $j$ by Lemma~\ref{lemma:central character}. Since $\H^4(\mathrm{GL}_2(5)) = 4 \times 8 \times 3$ has no five part, we find that $\H^4\big(5^{1+4}:\mathrm{GL}_2(5)\big)_{(5)}$, and hence also $\H^4(\rG_2(5))_{(5)}$, vanishes.
\end{proof}

Recall from Section~\ref{subsec:notation} that $\Omega_n(q)$ denotes the simple subquotient of the orthogonal group $\mathrm{O}_n(\bF_q)$, and that when $n\geq 5$ and $q$ is odd, $\Omega_n(q)$ is of index $2$ in $\mathrm{SO}_n(\bF_q)$. We will use the names $\Spin_n(q)$ and $2.\Omega_n(q)$ interchangeably.

\begin{Lemma}\label{lemma:O73} $\H_3(\Omega_7(3)) \cong \bZ/4$ and $\H_3(2.\Omega_7(3)) \cong \bZ/8$.
\end{Lemma}

\begin{proof} The criterion in Lemma~\ref{large primes} applies for the primes $p\geq 5$.
 The $2$-Sylow is contained in $\Sp_6(\bF_2)$, giving an upper bound of
 $\H^4(\Sp_6(\bF_2))_{(2)} = 2 \times 4$ for $\H_3(\Omega_7(3))$, and an upper bound of $\H^4(2\Sp_6(\bF_2))_{(2)} = 2 \times 8$ for $\H_3(\Spin_7(3))$, both from Lemma~\ref{lemma:Sp6}.

 Let $V$ denote the $105$-dimensional representation of $\Omega_7(3)$. It is a real representation. (Indeed, all representations of $\Omega_7(3)$ are real except for the two dual complex representations of degree~$1560$.) Conjugacy class $4\ra \in \Omega_7(3)$ acts on $V$ with trace~$-5$. Its square, conjugacy class~$2\rb$, acts with trace~$5$. It follows that $4\ra$ acts with spectrum $(+1)^{25} (-1)^{30} (i)^{25} (-i)^{25}$, and so the total Chern class of $V|_{\langle4\ra\rangle}$ is
 \begin{gather*} c(V)|_{\langle 4\ra\rangle} = (1+2t)^{30} (1+t)^{25}(1-t)^{25} = 1 - t^2 + \cdots.\end{gather*}
 In particular, $c_2(V)|_{\langle 4\ra\rangle}$ has order $4$, giving a lower bound of $4$ to the order of $c_2(V) \in \H^4(\Omega_7(3))$ and a lower bound of $8$ to the order of $\frac{p_1}2(V) \in \H^4(2.\Omega_7(3))$.

 To show that $\H_3(\Spin_7(3))_{(2)}$ is exactly $\bZ/8$ (which implies in turn that $\H_3(\Omega_7(3))_{(2)}$ is exactly $\bZ/4$) it suffices to give a class in $\H^4(2.\Sp_6(\bF_2))_{(2)}$ not in the image of restriction $\H_3(\Spin_7(3))_{(2)} \to \H^4(2.\Sp_6(\bF_2))_{(2)}$. We claim that the fractional Pontryagin class of the $15$-dimensional irrep of $\Sp_6(\bF_2)$ is such a class. (This representation is not Spin over $\Sp_6(\bF_2)$, but is Spin over $2\Sp_6(\bF_2)$. We will henceforth call its fractional Pontryagin class $\frac{p_1}2(15) \in \H^4(2\Sp_6(\bF_2))$.) To prove this, we consider the conjugacy classes $2\rb$ and $2\rd$ in $\Sp_6(\bF_2)$. They act on the $15$-dimensional irrep with traces $7$ and $-1$ respectively; equivalently, $2\rb$ acts with spectrum $1^{11} (-1)^4$ whereas $2\rd$ acts with spectrum $1^7 (-1)^8$. These two classes lift with order $2$ to $2\Sp_6(2)$. The fractional Pontryagin classes are therefore $\frac{p_1}2(15)|_{2\rb} = 1 \in \H^4(\langle 2\rb\rangle) \cong \bZ/2$ and $\frac{p_1}2(15)|_{2\rd} = 0$. But $2\rb$ and $2\rd$ both fuse to class $2\rc \in \Omega_7(3)$. It follows that $\frac{p_1}2(15) \in \H^4(2.\Sp_6(\bF_2))$ is not the restriction of any class in $\H^4(\Spin_7(3))$.

 It remains to handle the prime $p=3$. In general,
 the $p$-Sylow in a characteristic-$p$ group of Lie type is the nilpotent subgroup, and so is contained in any parabolic. We will use two maximal parabolics of
 the algebraic group
 $\mathrm{SO}_7$,
 corresponding to the Dynkin subdiagrams $B_2 \subseteq B_3$ and $A_1 \times A_1 \subseteq B_3$. These lead to two maximal subgroups of $\Omega_7(3)$ that contain the $3$-Sylow:
 \begin{gather*} 3^5:\mathrm{SO}_5(\bF_3), \qquad 3^{1+6}_+\colon (2A_4 \times A_4).2. \end{gather*}
 There is one more maximal subgroup of $\Omega_7(3)$ containing the $3$-Sylow, corresponding to the Dynkin diagram inclusion $A_2 \subseteq B_3$, which we will not use in the present proof, but will use in the proof of Corollary~\ref{cor:3o73}.

 The spectral sequence for $3^5:\mathrm{SO}_5(\bF_3)$ has $E_2$ page:
 \begin{gather*} \begin{array}{ccccc}
 3 \\
 0 & 0 & 3 \\
 0 & 3 & 0 \\
 0 & 0 & 0 & 0 & 0 \\
 \bZ & 0 & 0 & 2 & 2^2 \times 4 \times 3
 \end{array} \end{gather*}
 The bottom line was computed in HAP, and the middle entries in Cohomolo. The entry $E_2^{04} = 3$ corresponds to the symmetric pairing on~$3^5$.

We claim that the maps $\H^4(\Omega_7(3))_{(3)} \to \H^4\big(3^5:\mathrm{SO}_5(\bF_3)\big)_{(3)}$ and $\H^4(\mathrm{SO}_5(\bF_3))_{(3)} \to \H^4\big(3^5:\mathrm{SO}_5(\bF_3)\big)$ have trivial intersection. To see this, first note that $\mathrm{SO}_5(\bF_3) \cong \mathrm{Weyl}(\rE_6)$ has a~$6$-dimensional irrep, on which the conjugacy class $3\rc$ acts with trace~$3$. It follows that \begin{gather*}c_2(\text{6-dim irrep})|_{\langle 3\rc\rangle} \neq 0.\end{gather*}
But $3\rc \in \mathrm{SO}_5(\bF_3)$ has among its preimages in $3^5:\SO_5(\bF_3)$ one which fuses to class $3\rb \in \Omega_7(3)$, and $3\rb$ also meets $3^5 \subseteq 3^5:\mathrm{SO}_5(\bF_3)$. It follows that $c_2(\text{6-dim irrep}) \in \H^4(\mathrm{SO}_5(\bF_3)),$ when pulled back along $3^5:\mathrm{SO}_5(\bF_3) \to \mathrm{SO}_5(\bF_3)$, distinguishes conjugate-in-$\Omega_7(3)$ elements, and so is not the restriction of a class in $\H^4(\Omega_7(3))$.

 Since $\H^4(\Omega_7(3))_{(3)} \subseteq \H^4\big(3^5:\SO_5(\bF_3)\big)$ and the latter is an extension of a quotient of $\H^4(\SO_5(\bF_3))$ and a subspace of $\H^0\big(\mathrm{SO}_5(\bF_3); \H^4\big(3^5\big)\big)$, and since $\H^4(\Omega_7(3))_{(3)}$ does not meet $\H^4(\SO_5(\bF_3))$, the restriction map $\H^4(\Omega_7(3))_{(3)} \to \H^0\big(\mathrm{SO}_5(\bF_3); \H^4\big(3^5\big)\big)$ must be an injection. The order-$3$ conjugacy classes in $\Omega_7(3)$ that meet $3^5 \subset 3^5: \SO_5(\bF_3)$ are classes $3\ra$, $3\rb$, and $3\rc$. Specifically, the intersection of conjugacy class $3\ra$ and $3^5 \cong \bF_3^5$ consists of the nonzero vectors of norm~$0$, and the intersections of $3\rb$ and $3\rc$ with $3^5$ are the vectors of norm~$\pm 1$. (These are the three nontrivial $\SO_5(\bF_3)$-orbits in $\bF_3^5$.)
 The nonzero classes in $\H^0\big(\mathrm{SO}_5(\bF_3); \H^4\big(3^5\big)\big) \cong \bZ/3$ corresponds to the symmetric pairing and its negation, and so restrict trivially to $\langle 3\ra\rangle$ but nontrivially to $\langle 3\rb\rangle$ and $\langle 3\rc\rangle$. In particular the restriction map $\H^4(\Omega_7(3))_{(3)} \to \H^4(\langle 3\rb\rangle)$ is an injection.

The other maximal subgroup we consider is the one of shape $3^{1+6}_+:(2A_4 \times A_4).2$. It is the normalizer of conjugacy class $3\ra$. GAP can work with $\Omega_7(3)$ by using its faithful degree-$351$ permutation representation, and find this subgroup. In particular, GAP finds that the action of $(2A_4 \times A_4).2$ on $3^6$ is generated by the following three matrices
 \begin{gather*}
 \begin{pmatrix}
 . & 1 & . & . & 2 & 2\\
 2 & . & . & 2 & . & 1\\
 1 & 1 & . & 1 & 2 & 2\\
 . & 2 & 2 & . & . & .\\
 . & . & 1 & 1 & . & 2\\
 2 & . & 2 & 1 & 2 & 2
\end{pmatrix},
 \qquad
\begin{pmatrix}
 . & . & 2 & 2 & . & 1\\
 . & 1 & 2 & . & 1 & 1\\
 . & 2 & . & . & 2 & .\\
 1 & . & 1 & 2 & . & 1\\
 . & 2 & 2 & . & . & .\\
 . & 2 & . & . & 1 & .
\end{pmatrix},
 \qquad
\begin{pmatrix}
 2 & . & 2 & 1 & . & 2\\
 . & . & 2 & . & 2 & 2\\
 . & . & . & . & 1 & .\\
 2 & . & 1 & . & . & .\\
 . & 1 & 2 & . & 1 & 1\\
 . & 2 & . & . & 1 & .
\end{pmatrix}.
 \end{gather*}

 Recall from Lemma~\ref{lem:extraspecial-odd} that $\H^4\big(3^{1+6}\big) \cong \Sym^2\big(3^6\big) \oplus \big(\Alt^3\big(3^6\big)/3^6\big)$. The group $(2A_4 \times A_4).2$ contains the matrix $-1$, and which acts by $-1$ on $\big(\Alt^3\big(3^6\big)/3^6\big)$. Furthermore, the representation $3^6$ is not symmetrically self-dual. (In fact it is not self-dual: its antisymmetric pairing changes by a sign under the odd elements of $(2A_4 \times A_4).2$.) It follows that $\H^0\big((2A_4 \times A_4).2; \H^4\big(3^{1+6}\big)\big) = 0$.

 But this means in particular that the restriction map $\H^4(\Omega_7(3)) \to \H^4\big(3^{1+6}\big)$ vanishes. Conjugacy class $3\rb \in \Omega_7(3)$ meets $3^{1+6}$. It follows that the restriction $\H^4(\Omega_7(3)) \to \H^4(\langle 3\rb\rangle)$ is the zero map. But we showed above that $\H^4(\Omega_7(3))_{(3)} \to \H^4(\langle 3\rb\rangle)$ is an injection. So $\H^4(\Omega_7(3))_{(3)} = 0$.
\end{proof}

The Chevalley group $\Omega_7(3)$ has an exceptional cover: its multiplier is $6$, whereas the multiplier of $\Omega_7(q)$ is generically $2 = \pi_1(\SO(7,\bC))$.

\begin{Corollary} \label{cor:3o73}
 $\H_3(3.\Omega_7(3)) \cong \bZ/12$ and $\H_3(6.\Omega_7(3)) \cong \bZ/24$.
\end{Corollary}
\begin{proof} We must calculate $\H^4(3.\Omega_7(3))_{(3)} \cong \H^4(6.\Omega_7(3))_{(3)}$. It either vanishes or is $\bZ/3$ by Lemmas~\ref{lemma:schur cover} and~\ref{lemma:O73}.

 We used two of the three maximal parabolics of $\Omega_7(3)$ in the proof of Lemma~\ref{lemma:O73}; for this calculation we will use the third one, of shape $3^{3+3}:\SL(3,3)$.
 For the remainder of the proof we will call this subgroup $S$. Since $S$ contains the $3$-Sylow, the extension $3.\Omega_7(3)$ restricts to a~nontrivial central extension $3.S$.
 One can show, for instance by running a LHS spectral sequence, that $\H^1(S;3) = 0$ and $\H^2(S;3) = 3$. In particular, there is a unique nonsplit extension $3.S$
 up to isomorphism. (The two nonzero classes in $\H^2(S;3)$ are related by the outer automorphism of~$\bZ/3$.)

 By Lemma~\ref{lemma:O73}, $\H^4(\Omega_7(3))_{(3)} = 0$, and so $\H^4(3.\Omega_7(3))_{(3)} = \coker\big(\H^4(\Omega_7(3)) \to \H^4(3.\Omega_7(3))\big)$. This is in turn isomorphic to $\coker\big(\H^4(S) \to \H^4(3.S)\big)$ by Lemma~\ref{lemma:cokers}.

 The smallest complex representations of $\Omega_7(3)$ and $3.\Omega_7(3)$ have dimensions $78$ and $27$ respectively, equal to the smallest representations of the simple Lie group $\rE_6^{\rm adj}(\bC)$ and its simply connected cover $\rE_6^{\rm sc} = 3.\rE_6^{\rm adj}$. However, $\Omega_7(3)$ does not preserve the Lie bracket on the $78$-dimensional representation. It does preserve a lattice, and preserves the Lie bracket ``modulo~$2$'': in fact, $\Omega_7(3)$ embeds into the twisted Chevalley group $^2\rE_6(2) \subseteq \rE_6^{\rm adj}(\bF_4)$. The finite subgroups of the Lie group $\rE_6^{\rm adj}(\bC)$ {not already contained in a smaller Lie group} were classified in~\cite{MR1416728}.
 {In particular, $\rE_6^{\rm adj}(\bC)$ contains a subgroup isomorphic to $S$, lifting to the nonsplit extension $3.S \subseteq \rE_6^{\rm sc}(\bC)$. (Presumably $S$ is precisely the intersection of $\rE_6^{\rm adj}(\bC)$ and $\Omega_7(3)$ in their common 78-dimensional representation.)}

 Both $\H^4\big(B\rE_6^{\rm adj}\big)$ and $\H^4\big(B\rE_6^{\rm sc}\big)$ are infinite cyclic, but the restriction map $\H^4\big(B\rE_6^{\rm adj}\big) \to \H^4\big(B\rE_6^{\rm sc}\big)$ is not an isomorphism: its cokernel has order~$3$. By Lemma~\ref{lemma:cokers}, this forces the inclusion $\H^4(S) \to \H^4(3.S)$ to have cokernel of order~$3$.
\end{proof}

The Chevalley group $\mathrm{G}_2(3)$ has an exceptional multiplier of order~$3$.

\begin{Corollary} $\H_3(3.\mathrm{G}_2(3))$ has order $72$.
\end{Corollary}
\begin{proof} The $7$-dimensional representation of $\mathrm{G}_2$ provides an inclusion $\mathrm{G}_2(3) \subseteq \Omega_7(3)$, and the exceptional triple cover of $\Omega_7(3)$ restricts to the exceptional triple cover of $\mathrm{G}_2(3)$. Lemmas~\ref{lemma:schur cover} and~\ref{lemma:cokers} then force the inclusion $\H^4(\mathrm{G}_2(3)) \to \H^4(3.\mathrm{G}_2(3))$ to have cokernel of order~$3$.
\end{proof}

\section{Mathieu groups} \label{sec:Mathieu}

The low-degree homology groups of all Mathieu groups can be computed in HAP, and are listed in~\cite{SE09}, where details of HAP's implementation are discussed. That paper was the first to compute $\H_3(\rM_{24})$, and was able to compute up to $\H_4$ exactly for all Mathieu groups, and $\H_5$ exactly for all Mathieu groups except $\rM_{24}$, for which the $2$-part was left ambiguous. In \cite{GPRV} it is shown that the restriction map $\H^4(\rM_{24}) \to \H^4(\langle 12\rb\rangle)$ is an isomorphism, and that $\H^4(\rM_{24})$ is generated by the ``gauge anomaly'' of ``$\rM_{24}$ moonshine''. In \cite[Theorems~5.1 and~5.2]{jft} we gave direct proofs of the results $\H^4(\rM_{23}) = 0$ and $\H^4(\rM_{24}) = 12$ following the method outlined in Section~\ref{sec:methods}, and we also recognized that the generator of $\H^4(\rM_{24})$ from~\cite{GPRV} is more simply described as the fractional Pontryagin class of the defining degree-$24$ permutation representation. We remark that the same holds for $\rM_{11}$:

\begin{Proposition} $\H^4(\rM_{11}) \cong \bZ/8$ is generated by the fractional Pontryagin class of the defining degree-$11$ permutation representation.
\end{Proposition}
\begin{proof} Let $\Perm$ denote the permutation representation of $\rM_{11}$. The two conjugacy classes of order $8$ in $\rM_{11}$ have the same spectrum on $\Perm$: they act by $\operatorname{diag}\big(1,1,1,\zeta,i,\zeta^3,-1,-1,\zeta^{-3},-i,\allowbreak \zeta^{-1}\big)$, where $\zeta = \exp(2\pi i/8)$.
 Let $t \in \H^2(\bZ/8)$ denote a generator of $\H^\bullet(\bZ/8) = \bZ[t]/(8t)$.
 The total Chern class of $\Perm$, restricted to a cyclic group of order $8$, is therefore
 \begin{gather*} c(\Perm)|_{\langle 8\ra\rangle} = 1^3 (1-t)(1-2t)(1-3t)(1-4t)^2(1+3t)(1+2t)(1+t) = 1 + 2t^2 + \cdots.\end{gather*}
 In particular, $c_2(\Perm)$ has order divisible by~$4$. But $\Perm$ is a real and (since $\H_2(\rM_{11})$ vanishes) therefore Spin representation, and so $\frac{p_1}2(\Perm)$ has order divisible by~$8$.
\end{proof}

The Schur cover of $\rM_{12}$ is studied in \cite{CLW}; they compute $\H^4(2\rM_{12}) = 8^2 \times 3$ with HAP, and show that the map $\H^4(2\rM_{12}) \to \prod\limits_{g\in 2\rM_{12}} \H^4(\langle g\rangle)$ has kernel of order $2$. To fully describe $\H^4(2\rM_{12})$ requires moving slightly beyond cyclic groups, and also requires some notation. Let $\Perm$ denote (a choice of either) degree-$12$ permutation representation of $\rM_{12}$, and write $V_{12}$ for the unique $12$-dimensional faithful irrep of $2\rM_{12}$. Then $V_{12}$ is real, and hence spin, as a~$2\rM_{12}$-module (since~$2\rM_{12}$ has no central extensions). $\Perm \otimes \bR$ is not Spin as an $\rM_{12}$-module, but is automatically Spin as a $2\rM_{12}$-module, since $\H^2(2\rM_{12};\bZ/2) = 0$. Write $\frac{p_1}2(\Perm)$ and $\frac{p_1}2(V_{12})$ for their fractional Pontryagin classes. The group~$2\rM_{12}$ has two conjugacy classes of elements of order~$3$: class $3\rb$ acts on $\Perm$ with cycle structure~$3^{4}$. There are also four conjugacy classes of elements of order~$8$. Classes $8\ra$ and $8\rb$ differ by the central element and act on $\Perm$ with cycle structure~$1^2 2^1 8^1$; classes $8\rc$ and $8\rd$ differ by the central element and act with cycle structure~$4^1 8^1$. Finally, there is a unique conjugacy class of quaternion subgroups $Q_8 \subseteq 2\rM_{12}$ in which the center of $Q_8$ maps to the center of $2\rM_{12}$.

\begin{Proposition} $\H^4(2\rM_{12})$ is spanned by the classes $\frac{p_1}2(\Perm)$ and $\frac{p_1}2(V_{12})$. The restriction map \begin{gather*}\H^4(2\rM_{12}) \to \H^4(Q_8) \times \H^4(\langle 8\ra \rangle) \times \H^4(\langle 8\rc \rangle) \times \H^4(\langle 3\rb\rangle) \cong 8^3 \times 3\end{gather*} is an injection.
\end{Proposition}

We remark that the outer automorphism of $2\rM_{12}$ switches the two degree-$12$ permutation representations and also switches $8\ra\rb$ with $8\rc\rd$.

\begin{proof} We choose the following generators of $\H^4(Q_8)$ and $\H^4(\langle 8\ra\rangle) \cong \H^4(\langle 8\rc \rangle) \cong \H^4(C_8)$ and $\H^4(\langle 3\rb \rangle) \cong \H^4(C_3)$: the generator of $\H^4(Q_8)$ is the fractional Pontryagin class of the $4$-dimensional real representation (equal to the negative second Chern class of the $2$-dimensional complex irrep); if $n$ divides $24$, we take the unique generator of $\H^4(C_n)$ which is a cup square (it is unique by what Conway and Norton call ``the defining property of $24$'' \cite{MR554399}).

 It is straightforward to compute the images of $\frac{p_1}2(\Perm)$ and $\frac{p_1}2(V_{12})$ to $\H^4(Q_8) \times \H^4(\langle 8\ra \rangle) \times \H^4(\langle 8\rc \rangle) \times \H^4(\langle 3\rb\rangle)$. They are
 \begin{gather*}
 \tfrac{p_1}2(\Perm) \mapsto (3,1,1,-1), \\
 \tfrac{p_1}2(V_{12}) \mapsto (4,-1,1,-1).
 \end{gather*}
 But $(3,1,1,-1)$ and $(4,-1,1,-1)$ together generate a subgroup isomorphic to $8^2 \times 3 \cong \H^4(2\rM_{12})$ inside $8^3 \times 3$.
\end{proof}

The covers of $\rM_{22}$ are not directly computable by HAP, since they do not have sufficiently small permutation representations.

\begin{Proposition} \label{prop:mathieu}
 The covers of $\rM_{22}$ have the following third homology groups:
 \begin{gather*} \H_3(2\rM_{22}) = 4, \qquad \H_3(3\rM_{22}) = 3, \qquad \H_3(4\rM_{22}) = 8, \\ \H_3(6\rM_{22}) = 12, \qquad \H_3(12\rM_{22}) = 24.\end{gather*}
\end{Proposition}
\begin{proof} Given Lemma~\ref{lemma:schur cover} together with the computer computation $\H_3(\rM_{22}) = 0$, it suffices to give lower bounds $\H^4(2\rM_{22}) \geq 4$, $\H^4(4\rM_{22}) \geq 8$, and $\H^4(3\rM_{22}) \geq 3$. The first two can be handled simultaneously as follows. $2\rM_{22}$ has a unique faithful $210$-dimensional irrep $V$. Conjugacy class $4\rc \in 2\rM_{22}$ acts on $V$ with spectrum $1^{50} (-1)^{50} i^{55} (-i)^{55}$. Let $t \in \H^2(\langle 4\rc\rangle)$ denote a generator. Then
 \begin{gather*} c_2(V)|_{\langle 4\rc\rangle} = t^2 \in \H^4(\langle 4\rc\rangle) \end{gather*}
 has order $4$. This gives the lower bound for $\H^4(2\rM_{22})$. The representation $V$ is real, but it is not Spin as a $2\rM_{22}$-module (since, indeed, $c_2(V)$ is not divisible by $2$). It is, however, Spin as a $4\rM_{22}$-module. Since $\H^4(2\rM_{22}) \to \H^4(4\rM_{22})$ is an injection, $c_2(V)|_{4\rM_{22}}$ has order divisible by $4$, so $\frac{p_1}2(V) \in \H^4(4\rM_{22})$ has order divisible by $8$. This provides the lower bound for $\H_3(4\rM_{22})$. Finally, for $3\rM_{22}$, we may use either $21$-dimensional faithful representation $W$. Element $3\rc \in 3\rM_{22}$ acts on $W$ with trace $0$, and so $c_2(W)|_{\langle 3\rc\rangle} = -\frac{21}3t^2 \neq 0 \in \H^4(\langle 3\rc\rangle)$. Thus $c_2(W)$ has order divisible by $3$ in $\H^4(3\rM_{22})$.
\end{proof}

\section{Leech lattice groups} \label{sec:Leech}

\subsection{Higman--Sims group} \label{sub:HS}

The smallest faithful permutation representations of the Higman--Sims group $\mathrm{HS}$ and its double cover $2\mathrm{HS}$ have degrees $100$ and $704$ respectively.
Using these representations, we find that HAP can compute $\H_3(\mathrm{HS})$ and $\H_3(\mathrm{2HS})$ without further human assistance:
\begin{gather*} \H_3(\mathrm{HS}) \cong (\bZ/2)^2, \qquad \H_3(2\mathrm{HS}) \cong \bZ/2 \times \bZ/8.\end{gather*}

Since $\H^4(G) \subseteq \H^4(2G)$ has index $4$, the latter must contain elements with nontrivial restriction to the central $2 \subseteq 2G$ by Lemma~\ref{lemma:schur cover}.
 It is not hard to check that all complex representations~$V$ of $2\mathrm{HS}$ have $c_2(V)|_{2} = 0$, and all real representations have $\frac{p_1}2(V)|_2 = 0$. In particular, we do not know generators for $\H^4(2\mathrm{HS}) \cong \bZ/2 \times \bZ/8$.

\subsection{Janko group 2}\label{sub:J2}

The smallest permutation representations of Janko's second group $\rJ_2$ (also called the Hall--Janko group $\mathrm{HJ}$) and of its double cover $2\rJ_2$ have degrees $100$ and $200$ respectively, and HAP computes:
\begin{gather*} \H_3(\rJ_2) \cong \bZ/30, \qquad \H_3(2\rJ_2) \cong \bZ/120. \end{gather*}

We record some finer information in this section: in particular we show that $\H^4(2\rJ_2;\bZ)$ is generated by the Chern class of either six-dimension irreducible representation of $2\rJ_2$, and the outer automorphism of $2\rJ_2$ acts by multiplication by $49$. For this, we will compute the Chern classes of the irreducible representations of $\SL(2,5)$, which is a subgroup of $2\rJ_2$ in two non-conjugate ways.

Let $\pi$ be a two-dimensional irreducible representation of $C = \SL(2,5)$. Then $\pi$ is faithful and has trivial determinant, and we can use the McKay correspondence to parametrize the other irreps of $G$, by nodes in the extended Dynkin diagram of type $\mathrm{E}_8$. That parametrization is
\begin{gather*}
\begin{tikzpicture}[anchor=base]
 \path
 node (1a) {$1$,}
 ++(-1,0) node (2a) {$\pi$}
 ++(-1.5,0) node (3a) {$S^2(\pi)$}
 ++(-1.75,0) node (4a) {$S^3(\pi)$}
 ++(-1.75,0) node (5a) {$S^4(\pi)$}
 ++(-1.75,0) node (6a) {$S^5(\pi)$}
 +(0,1) node (3b) {$S^2(\pi^\circ)$}
 ++(-1.75,0) node (4b) {$\pi \otimes \pi^\circ$}
 ++(-1.5,0) node (2b) {$\pi^\circ$}
 ;
 \draw (1a.west |- 2a.east) -- (2a.east);
 \draw (2a.west) -- (3a.east |- 2a.west);
 \draw (3a.west |- 2a.west) -- (4a.east |- 2a.west);
 \draw (4a.west |- 2a.west) -- (5a.east |- 2a.west);
 \draw (5a.west |- 2a.west) -- (6a.east |- 2a.west);
 \draw (6a.west |- 2a.west) -- (4b.east |- 2a.west);
 \draw (4b.west |- 2a.west) -- (2b.east |- 2a.west);
 \draw (6a) -- (3b);
\end{tikzpicture}
\end{gather*}
where we have written $S^n(-)$ as short for the $n$th symmetric power of a representations, and $\pi^{\circ}$ for the image of $\pi$ under a nontrivial outer automorphism of $G$.

\begin{Lemma}\label{mckay e8}$\H^4(\SL(2,5);\bZ) \cong \bZ/120$. If $\pi$ is an irreducible representation of dimension $2$ then $c:= c_2(\pi)$ generates $\H^4(G;\bZ)$, and the Chern classes of the remaining irreducibles are
\begin{gather*}
\begin{tikzpicture}[anchor=base]
 \path
 node (1a) {$0$.}
 ++(-1,0) node (2a) {$c$}
 ++(-1,0) node (3a) {$4c$}
 ++(-1.25,0) node (4a) {$10c$}
 ++(-1.25,0) node (5a) {$20c$}
 ++(-1.25,0) node (6a) {$35c$}
 +(0,1) node (3b) {$76c$}
 ++(-1.5,0) node (4b) {$100c$}
 ++(-1.25,0) node (2b) {$49c$}
 ;
 \draw (1a.west |- 2a.east) -- (2a.east);
 \draw (2a.west) -- (3a.east |- 2a.west);
 \draw (3a.west |- 2a.west) -- (4a.east |- 2a.west);
 \draw (4a.west |- 2a.west) -- (5a.east |- 2a.west);
 \draw (5a.west |- 2a.west) -- (6a.east |- 2a.west);
 \draw (6a.west |- 2a.west) -- (4b.east |- 2a.west);
 \draw (4b.west |- 2a.west) -- (2b.east |- 2a.west);
 \draw (6a) -- (3b);
\end{tikzpicture}
\end{gather*}
\end{Lemma}

\begin{proof}The integer cohomology ring of $B\SU(2)$ is isomorphic to $\bZ[c]$, where~$c$ in degree~$4$ is~$c_2$ of the tautological representation $V$ of $\SU(2)$, and the integer cohomology ring of $B\mathrm{U}(1)$ is isomorphic to $\bZ[b]$ where $b$ in degree~$2$ denotes the first Chern class of the tautological one-dimensional representation. The restriction of the $n$th symmetric power of $V$ to a maximal torus $\mathrm{U}(1) \subset \mathrm{SU}(2)$ splits as a sum of $1$-dimensional representations of weights
\begin{gather*}
n, (n-2),(n-4),\ldots,(-n+2),(-n).
\end{gather*}
Thus, the total Chern class of $S^n(V)$ can be written as
\begin{gather*}
c_t(S^n(V)) := \big(1+c_2(S^n(V))t^2 + c_4(S^n)t^4 + \cdots\big) = (1+nbt)(1+(n-2)bt)\cdots(1-nbt).
\end{gather*}
In particular
\begin{gather*}
c_2(V) = -b^2, \qquad c_2\big(S^2(V)\big) = -4b^2, \qquad c_2\big(S^3(V)\big) =-10b^2, \\ c_2\big(S^4(V)\big) = -20b^2, \qquad c_2\big(S^5(V)\big) = -35b^2.
\end{gather*}
This explains six of the nine Chern classes reported in the Lemma. It remains to compute $c_2(\pi^\circ)$, $c_2\big(S^2(\pi^\circ)\big)$ and $c_2(\pi \otimes \pi^\circ)$. We will need a concrete description of the outer automorphism of~$\SL(2,5)$.

Let $o\colon \SL(2,5) \to \SL(2,5)$ denote the conjugation by the diagonal matrix
\begin{gather*}
\left(
\begin{matrix}
2 \\ & 1
\end{matrix}
\right) \in \GL(2,5).
\end{gather*}
Then $\pi^{\circ}$ is the composite of $\pi$ with $o$. We claim that $o$ acts as multiplication by $49$ on $\H^4(\SL(2,5),\bZ)$. To see this, make the following observations:
\begin{enumerate}\itemsep=0pt
\item There are six subgroups of order $5$ in $\SL(2,5)$. All of them are conjugate to each other and two of them are preserved by $o$.
Writing $H_5$ for either one of these two $o$-fixed cyclic subgroups, the action of~$o$ on~$\H^2(H_5;\bZ) \cong \bZ/5$ is multiplication by a primitive $4$th root of~$1$ (in $\bF_5$)~-- the action on $\H^4(H_5;\bZ) = \H^2(H_5;\bZ)^{\otimes 2}$ is therefore by~$-1$.
\item There are five subgroups of order $24$ in $\SL(2,5)$, all of them conjugate to each other. Exactly one of them~-- call it $H_{24}$~-- is preserved by $o$. The action of $o$ on $H_{24}$ coincides with the conjugation action of an element $x \in \SU(2)$, in particular $o$ acts as the identity on $\H^4(H_{24};\bZ) \cong \bZ/24$.
\end{enumerate}
The restriction maps $\H^4(\SL(2,5);\bZ) \to \H^4(H_5;\bZ) \oplus \H^4(H_{24};\bZ)$ is an isomorphism. The number~$49$ arises as the unique solution to $49 = -1 \pmod {5}$ and $49 = 1 \pmod {24}$. It follows that $c_2(\pi^{\circ}) = 49 c$ and $c_2(S^2(\pi^\circ)) = 49 \cdot 4c = 76c$. To compute $c_2(\pi \otimes \pi^{\circ})$, we note that $\pi \otimes \pi^{\circ}$ is isomorphic to its conjugate by $o$, so its restriction to $\H^4(H_5;\bZ)$ vanishes, while $\pi \otimes \pi^{\circ}$ and $1+ \Sym^2(\pi)$ have the same restriction to $H_{24}$; the number $100$ arises as the unique solution to $100 = 0 \pmod {5}$ and $100 = 4 \pmod {24}$.
\end{proof}

\begin{Proposition} \label{lem:J2} $\H^4(2\rJ_2;\bZ)\cong \bZ/120$ is generated by $c_2(V)$, where $V$ is either $6$-dimensional irrep of $2\rJ_2$. For one $($but not the other$)$ of the two conjugacy classes of $\SL(2,5)$-subgroups of~$2\rJ_2$, the restriction map $\H^4(2\rJ_2;\bZ) \to \H^4(\SL(2,5);\bZ)$ is an isomorphism. The outer automorphism of $2\rJ_2$ acts by multiplication by $49$ on $\H^4(2\rJ_2;\bZ)$.
\end{Proposition}
\begin{proof}

Let $V$ and $V'$ denote the two $6$-dimensional irreps of $2\rJ_2$. They are exchanged by the outer automorphism of $2\rJ_2$. Let us write $\SL(2,5)_a \subseteq 2\rJ_2$ and $\SL(2,5)_b \subseteq 2\rJ_2$ for representatives of the two conjugacy classes of $\SL(2,5)$-subgroups. The representations $V$ and $V'$ restrict as
 \begin{gather*} V|_{\SL(2,5)_a} \cong \pi \oplus 2\pi^{\circ}, \qquad V'|_{\SL(2,5)_a} \cong \pi^{\circ} \oplus 2\pi, \\
 V|_{\SL(2,5)_b} \cong \pi \oplus S^3(\pi), \qquad V'|_{\SL(2,5)_b} \cong \pi^{\circ} \oplus S^3(\pi).\end{gather*}
 It follows that the outer automorphism of $2\rJ_2$ restricts along either embedding $\SL(2,5) \subseteq 2\rJ_2$ to the outer automorphism of $\SL(2,5)$. Moreover, $c_2(\pi \oplus 2\pi^{\circ}) = c_2(\pi) + 2c_2(\pi^{\circ}) = 99c$, which has order $40$ in $\bZ/120$, but $c_2\big(\pi \oplus S^3(\pi)\big) = c_2(\pi) + c_2(S^4(\pi)) = 11 c$ has order $120$. Thus $\H^4(2\rJ_2;\bZ) \to \H^4(\SL(2,5)_b;\bZ) \cong \bZ/120$ is surjective, and hence an isomorphism
 given the HAP computation. Since the outer automorphism acts by multiplication by $49$ on $\H^4(\SL(2,5)_b;\bZ)$, it must also act by multiplication by $49$ on $\H^4(2\rJ_2;\bZ)$.
\end{proof}

\subsection{Conway groups}

In \cite[Theorems 0.1 and 5.3]{jft} we showed that
\begin{gather*} \H^4(\Co_1) \cong \bZ/12, \qquad \H^4(2.\Co_1) \cong \bZ/24,\end{gather*}
and that these groups are generated by the fractional Pontryagin classes of the $276$- and $24$-dimensional representations, respectively. Let us denote the $24$-dimensional real representation of $2.\Co_1$ by the name $\Leech$. The second and third Conway groups $\Co_2$ and $\Co_3$ are subgroups of~$2.\Co_1$, and so $\Leech$ restricts to representations of each (where it splits as a trivial representation plus a $23$-dimensional irrep).

\begin{Theorem}\label{thm.co2} $\H^4(\Co_2) \cong \bZ/4$ is generated by the restriction of $\frac{p_1}2(\Leech)$.
\end{Theorem}

\begin{proof} In \cite[Theorem 7.1]{jft} we gave a formula for $\frac{p_1}2(\Leech)|_{\langle g\rangle}$ for all elements $g \in 2.\Co_1$ in terms of the Frame shape of $g$ in the $\Leech$ representation. (Introduced by Frame in~\cite{MR0269751} to study the $\rE_8$ Weyl group, Frame shapes encode the characteristic polynomials of lattice-preserving orthogonal matrices.)
 The conjugacy class $4\rg \in \Co_2$ has Frame shape $4^6$, and so $\frac{p_1}2(\Leech)$ restricts with order $4$ to this conjugacy class. This gives the lower bound $\H^4(\Co_2) \geq 4$.

For the upper bound, Lemma~\ref{large primes} handles the primes $\geq 7$. For the primes $3$ and $5$, we note that $\Co_2$ contains a subgroup isomorphic to $\mathrm{McL}$, which in turn contains the $3$- and $5$-Sylows. Since $\H^4(\mathrm{McL}) = 0$ (see Section~\ref{sec:McL}), we learn that $\H^4(\Co_2)_{(p)} = 0$ for $p$ odd.

 It remains to give an upper bound for the $2$-part of $\H^4(\Co_2)$. The $2$-Sylow in $\Co_2$ is contained in a subgroup isomorphic to $2^{10}:(\mathrm{M}_{22}:2)$, where $E = 2^{10}$ is an irreducible $(\rM_{22}:2)$-module over $\bF_2$. The subgroup $E$ contains elements with Frame shape~$2^{12}$. By \cite[Theorem~7.1]{jft}, $\frac{p_1}2(\Leech)$ restricts nontrivially to such an element, and so $\frac{p_1}2(\Leech)|_E \in \H^0\big(\rM_{22}:2; \H^4(E)\big)$ is nonzero. There are two irreducible $10$-dimensional $(\rM_{22}:2)$-modules over $\bF_2$, which we will call $V_a$ and $V_b$, where the letters ``$a$'' and ``$b$'' match the notation in~\cite{ATLASonline}. They enjoy
 \begin{gather*}
 \H^0(\mathrm{M}_{22}:2; V_a) = \H^0\big(\mathrm{M}_{22}:2; \Alt^2(V_a)\big) = \H^0\big(\mathrm{M}_{22}:2; \Alt^3(V_a)\big) = 0, \\
 \H^0(\mathrm{M}_{22}:2; V_b) = \H^0\big(\mathrm{M}_{22}:2; \Alt^2(V_b)\big) = 0, \qquad \H^0\big(\mathrm{M}_{22}:2; \Alt^3(V_b)\big) \cong \bZ/2.
 \end{gather*}
 Since $\H^4(E) \cong E^*.\Alt^2(E^*).\Alt^3(E^*)$ by Lemma~\ref{elemab}, the only way for $\H^4(E)$ to have a nontrivial $(\mathrm{M}_{22}:2)$-fixed point is if $E^* \cong V_b$.

 With this isomorphism in hand, we can compute the $E_2$ page of the LHS spectral sequence for $E\colon (\rM_{22}:2)$:
 \begin{gather*}
 \begin{array}{ccccc}
 2 \\
 0 & 2 & 2^2 \\
 0 & 0 & 0 \\
 0 & 0 & 0 & 0 & 0 \\ \hdashline
 \bZ & 0 & 2 & 2 & 2^2
 \end{array}
 \end{gather*}
 The bottom row is computed by HAP, and the middle rows by Cohomolo. The dashed line reminds that the extension $E\colon (\rM_{22}:2)$ splits, and so $\H^4(\rM_{22}:2)$ is a direct summand of $\H^4(E\colon (\rM_{22}:2))$.

To complete the proof, it suffices to show that $\H^4(\Co_2) \to \H^4(E:\rM_{22}:2)$ and $\H^4(\rM_{22}:2) \to \H^4(E:\rM_{22}:2)$ have trivial intersection. There are three conjugacy classes of order $2$ in $\rM_{22}:2$, with cycle structures $1^6 2^8$, $1^8 2^7$, and $2^{11}$ in the degree-$22$ permutation representation. Together, these three classes detect $\H^4(\rM_{22}:2)$: if $\alpha \in \H^4(\rM_{22}:2)$ is nonzero, then there is an element $g \in\rM_{22}:2$ of order $2$ such that $\alpha|_{\langle g\rangle} \neq 0$. (Indeed, the images of $\H^4(2) \to \H^4(\rM_{22}:2)$ and $\H^4(\rM_{24}) \to \H^4(\rM_{22}:2)$ are transverse, and one can quickly compute the restrictions of their images to the three elements of order $2$.)

Given $\alpha \neq 0 \in \H^4(\rM_{22}:2)$, choose $g \in\rM_{22}:2$ of order $2$ such that $\alpha|_{\langle g\rangle} \neq 0$. Choose also an order-2 lift $\tilde g$ of $g$ in $E:\rM_{22}:2$, and let $\tilde\alpha \in \H^4(E:\rM_{22}:2)$ denote the pullback of $\alpha$. Then $\tilde\alpha|_{\langle\tilde g\rangle} = \alpha|_{\langle g\rangle} \neq 0$. But $\Co_2$ has only three conjugacy classes of order $2$, distinguished by their traces on $\Leech$, and all three classes meet $E$. Since $\tilde\alpha|_{E} = 0$, we find that $\tilde\alpha$ takes different values on conjugate-in-$\Co_2$ elements, and so cannot be the restriction of a class in $\H^4(\Co_2)$. This completes the proof that $\H^4(\Co_2) \cong \bZ/4$.
\end{proof}

\begin{Theorem}\label{thm.co3}
 $\H^4(\Co_3) \cong \bZ/6$ is generated by the restriction of $\frac{p_1}2(\Leech)$.
\end{Theorem}

\begin{proof}The conjugacy class $6\re \in \Co_3$ has Frame shape $6^4$ in the Leech representation. It follows from \cite[Theorem 7.1]{jft} that $\frac{p_1}2(\Leech)|_{\langle 6\re\rangle}$ has order $6$, giving the claimed lower bound $\H^4(\Co_3) \geq 6$.
Lemma~\ref{large primes} handles the primes $\geq 7$, and $\Co_3$ contains a copy of $\mathrm{McL}$, which contains the $5$-Sylow.

The $3$-Sylow in $\Co_3$ is contained in a subgroup of shape $3^{11}:(2 \times \rM_{11})$. There are two irreducible $11$-dimensional representations of $\rM_{11}$ over $\bF_3$, dual to each other. They lead to LHS spectral sequences with $E_2$ pages
\begin{gather*}
 \begin{array}{ccccc}
 0 \\
 0 & 0 & 0 \\
 0 & 0 & 0 \\
 0 & 0 & 0 & 0 & 0 \\
 \bZ & 0 & 2 & 0 & 2 \times 8
 \end{array}
 \qquad \text{and} \qquad
 \begin{array}{ccccc}
 0 \\
 0 & 3 & 3 \\
 0 & 0 & 0 \\
 0 & 0 & 0 & 0 & 0 \\
 \bZ & 0 & 2 & 0 & 2 \times 8
 \end{array}
\end{gather*}
Only the latter of these is consistent with the lower bound $\H^4(\Co_3) \geq 3$, and provides the desired upper bound {on $\H^4(\Co_3)_{(3)}$}.

To complete the proof we must verify that $\H^4(\Co_3)_{(2)} \leq 2$. The $2$-Sylow in $\Co_3$ is contained in three maximal subgroups: one the form $2^4 \cdot \GL_4(\bF_2)$, one of the form $2 \cdot \Sp_6(\bF_2)$, and one of order $2^{10}.3^3$. Lemma~\ref{lemma:Dempwolff}(2) and Lemma~\ref{lemma:Sp6} give
\begin{gather*}
 \H^4\big(2^4 \cdot \GL_4(\bF_2);\bZ\big) = \bZ/2 \oplus \bZ/4 \oplus \bZ/3, \\
 \H^4(2 \cdot \Sp_6(\bF_2);\bZ) = \bZ/2 \oplus \bZ/8 \oplus \bZ/3.
\end{gather*}
By Lemma~\ref{transfer restriction}, $\H^4(\Co_3)_{(2)}$ is a direct summand of both $\bZ/2 \oplus \bZ/4$ and of $\bZ/2 \oplus \bZ/8$, which forces $\H^4(\Co_3)_{(2)} \subseteq \bZ/2$.
\end{proof}

\subsection{McLaughlin group} \label{sec:McL}

HAP is able to directly compute
\begin{gather*} \H_3(\mathrm{McL}) = 0 \end{gather*}
 by using the permutation representation of degree $275$. HAP is unable to directly compute $\H_3(3\mathrm{McL})$ because the smallest faithful permutation representation of $3\mathrm{McL}$ has degree 66825. Lemma~\ref{lemma:schur cover} only provides the upper bound $\H_3(3\mathrm{McL}) \leq 3$.
Nevertheless, with some human involvement, we do have:

\begin{Theorem} \label{thm.mcl}
 $\H_3(3\mathrm{McL}) = 0$.
\end{Theorem}

\begin{proof}
 The computer calculation of $ \H_3(\mathrm{McL})$ leaves only the $3$-part of $\H_3(3\mathrm{McL})$ to be computed. But we can also dispense with the other parts directly.
 The $2$-Sylow in $\mathrm{McL}$ is contained in a maximal subgroup of shape $\rM_{22}$, and $\H_3(\rM_{22})_{(2)} = 0$ by computer calculation (see Proposition~\ref{prop:mathieu}). The $5$-Sylow is contained in a group of shape $5^{1+2}:3:8$, and HAP quickly computes \begin{gather*}\H^4\big(5^{1+2}:3:8\big)_{(5)} = 0.\end{gather*} (The ATLAS contains generators for most maximal subgroups of sporadic groups.)
 Lemma~\ref{large primes} handles the primes $p\geq 7$.

 Only the prime $3$ is left. The $3$-Sylow in $3\mathrm{McL}$ is contained in two maximal subgroups, one of shape $3^5.\rM_{10}$ and the other of shape $3^{2+4}:2S_5$.
 The latter is more useful, and for the remainder of the proof we will call it~$S$.
 The quotient $2S_5$ is the one listed in the ATLAS under the names ``$2S_5\mathrm{i}$'' and ``$\operatorname{Isoclinic}(2.A_5.2)$''; it is the group of shape $2S_5$ that contains elements of order $12$. This $2S_5$ has a faithful 4-dimensional representation over $\bF_3$. The quotient of $S$ in $\mathrm{McL}$ has shape $3^{1+4}:2S_5$, and the ``central'' $3$ is not central, but rather transforms by the sign representation of $2S_5$. In terms of the $4$-dimensional module, it corresponds to a symplectic form on $3^4$ which is $2A_5$- but not $2S_5$-fixed. There is also a symplectic form which is $2S_5$-fixed, and the $3^{2+4}$ subgroup of $S$ extends $3^4$ by both symplectic forms simultaneously.

 After a multi-hour computation, HAP reports
 \begin{gather*} \H_1(S; \bF_3) = \H_2(S; \bF_3) = 0, \qquad \H_3(S; \bF_3) = 3, \end{gather*}
 from which we learn that
 \begin{gather*} \H_1(S)_{(3)} = \H_2(S)_{(3)} = 0, \qquad \H_3(S)_{(3)} \text{ is cyclic.} \end{gather*}
 On the other hand,
 \begin{gather*} \H_3(2S_5)_{(3)} = 3, \end{gather*}
 and since the extension $S = 3^{2+4}:2S_5$ splits, $\H_3(S)_{(3)}$ contains $\H_3(2S_5)_{(3)}$ as a direct summand. Passing to cohomology, we learn that the pullback map
 \begin{gather*} \H^4(2S_5) \to \H^4(S) \end{gather*}
 is an isomorphism.

 There is a unique conjugacy class of order $3$ in $2S_5$, and the restriction map $\H^4(2S_5)_{(3)} \to \H^4(\langle 3\ra\rangle)$ is an isomorphism. Take any element $g \in 2S_5$ of order $3$ and lift it to an order-$3$ element $\tilde g \in S$. Then the composition $\H^4(2S_5)_{(3)} \overset\sim\to \H^4(S)_{(3)} \to \H^4(\langle \tilde g \rangle)$ is an isomorphism. On the other hand, all conjugacy classes of order-$3$ elements in $3\mathrm{McL}$ meet the normal subgroup $3^{2+4} \subseteq S$, and the composition $\H^4(2S_5)_{(3)} \overset\sim\to \H^4(S)_{(3)} \to \H^4\big(3^{2+4}\big)$ is zero.

 Thus the image of $\H^4(2S_5)_{(3)} \overset\sim\to \H^4(S)_{(3)}$ has trivial intersection with the restriction map $\H^4(3\mathrm{McL})_{(3)} \hookrightarrow \H^4(S)_{(3)}$, and so $\H^4(3\mathrm{McL})_{(3)} = 0$.
\end{proof}

\subsection{Suzuki group}\label{sec.suz}

The Schur cover of the Suzuki group is the beginning of a famous sequence of subgroups of $2\Co_1$ centralizing actions of binary alternating groups on the Leech lattice; $6\Suz$ centralizes an action of $2A_3 \cong \bZ/6$, which corresponds to a ``complex structure'' on the Leech lattice. In particular, $6\Suz$ has a conjugate pair of 12-dimensional irreducible complex representations, (either one of) which we will call $V$ throughout this section. The underlying real representation of $V$ is $(\Leech \otimes \bR)|_{6\Suz}$.

The maximal subgroups of $\Suz$ are listed in \cite{MR716777}. The 2-Sylow subgroup of $\Suz$ is contained in {the centralizer of~2a}, a subgroup of shape
\begin{gather*}
2^{1+6}\cdot\SWeyl (\rE_6) \subset \Suz,
\end{gather*}
where $\SWeyl(\rE_6)$ is the index-2 subgroup of the Weyl group that acts with trivial determinant of the reflection representation. The ATLAS calls this group $\SWeyl (\rE_6) = \rU_4(2)$. We write it as a Weyl group to make the action on $2^6$ transparent: it is the mod-$2$ reduction of the action of $\mathrm{Weyl}(\rE_6)$ on the $\rE_6$-lattice. In particular, the quadratic form on $2^6$ associated with the extraspecial group $2^{1+6}$ has Arf invariant $-1$.

\begin{Lemma}\label{lem:inE6}
Let $\mathrm{E}_6^{\mathrm{adj}}$ denote the compact Lie group of adjoint type $\mathrm{E}_6$. There is a homomorphism
\begin{gather*}
2^{1+6}\cdot \SWeyl(\rE_6) \to \mathrm{E}_6^{\mathrm{adj}},
\end{gather*}
whose kernel is the central $\bZ/2$ and which maps $2^{1+6}$ onto the $2$-torsion subgroup of a maximal torus in $\mathrm{E}_6^{\mathrm{adj}}$.
\end{Lemma}

\begin{proof}Let $T.W$ be the normalizer of a maximal torus in a compact Lie group, and $T[2]$ for the $2$-torsion of $T$. A subgroup $T[2].W \subset T.W$ is studied in \cite{MR0206117} and proved to be nonsplit for groups of type $\mathrm{E}_6$ in \cite{MR0376956}.
{The quotient of $2^{1+6}\cdot\SWeyl (\rE_6) \subset \Suz$ by the central $\bZ/2$ is also a nonsplit extension $2^6 \cdot \SWeyl(\rE_6)$, as can be quickly confirmed in GAP. (Indeed, GAP can easily look up the maximal subgroup $2^{1+6}.\SWeyl(\rE_6)$ of $\Suz$ in the ATLAS, compute its quotient $2^6.\SWeyl(\rE_6)$, and work out that there are no nontrivial homomorphisms into it from $\SWeyl(\rE_6)$.)}
We use Cohomolo to compute
\begin{gather*}
\H^2\big(\SWeyl(\rE_6),2^6\big) = \bZ/2,
\end{gather*}
so the two non-split extensions of $\SWeyl(\rE_6)$ by $2^6$ must be isomorphic.
\end{proof}

\begin{Remark}There is similar relationship between a centralizer in $\Co_1$ and the group $2^8\cdot\mathrm{Weyl}(\mathrm{E}_8)$ in the $\mathrm{E}_8$ Lie group, which has some conformal-field theoretic significance. The $\mathrm{E}_8$ Lie group acts on the $\mathrm{E}_8$-lattice VOA, and the simple group $\Co_1$ acts on Duncan's ``supermoonshine'' SVOA of \cite{MR2352133, JohnDuncan-thesis}. In~\cite{JohnDuncan-thesis}, the latter is constructed out of the former in such a way as to give a natural identification between subquotients of $\mathrm{E}_8$ and $\Co_1$ of shape $2^8 \cdot \mathrm{PSWeyl}(E_8)$. (Here $\mathrm{PSWeyl}$ denotes the quotient of $\SWeyl$ by the center, which is nontrivial for Weyl group of $\mathrm{E}_8$.) The homomorphism from Lemma~\ref{lem:inE6} can be constructed by starting with this identification and analyzing $\bZ/3$-centralizers (in $\mathrm{E}_8$ and in $\Co_1$). It would be interesting to find a direct ``moonshine'' construction of $\Suz$ from the $\mathrm{E}_6$ lattice making this homomorphism transparent.
\end{Remark}

Our goal in this section is to prove:

\begin{Theorem}\label{thm:suz} The Suzuki group and its Schur covers have the following fourth cohomology groups:
 \begin{gather*} \H^4(\Suz) = \bZ/4, \qquad \H^4(2\Suz) = \bZ/8, \qquad \H^4(3\Suz) = \bZ/12, \qquad \H^4(6\Suz) = \bZ/24.\end{gather*}
 $\H^4(6\Suz)$ is generated by $c_2(V)$, where $V$ denotes either $12$-dimensional complex irrep.
\end{Theorem}

We will split the proof into a series of lemmas.

\begin{Lemma}\label{lem:suz lower bound} $c_2(V)$ has order $24$ in $\H^4(6\Suz)$.
\end{Lemma}

\begin{proof} Since the underlying real representation of $V$ is $\Leech \otimes \bR$, we have
\begin{gather*} c_2(V) = -\tfrac{p_1}2(\Leech)|_{6\Suz}.\end{gather*} Then \cite[Theorem 0.1]{jft} gives an upper bound of $24$ on the order of $c_2(V)$.

The action of $6\Suz$ on the Leech lattice includes elements with Frame shape $3^8$; according to \cite[Theorem 7.1]{jft}, $\frac{p_1}2(\Leech)$ has nontrivial restriction to such elements. This gives a lower bound of $3$ on the order of $c_2(V)$.

$6\Suz$ contains a maximal subgroup of shape $6A_7$. As observed in \cite[Lemma 4.1]{jft}, there is a~unique conjugacy class of subgroups $D_8 \subseteq A_6$, where $D_8$ denotes the dihedral group of order~$8$. Along the standard inclusion $A_6 \subseteq A_7$, the $6$-fold cover pulls back to the cover $3 \times 2D_8$ of~$D_8$, where $2D_8$ denotes the binary dihedral group of order~$16$. This group $2D_8$ is the one used in \cite[Lemma~4.1]{jft}, where it is shown that $\frac{p_1}2(\Leech)|_{2D_8}$ has order $8$. This gives a lower bound of~$8$ on the order of~$c_2(V)$.
\end{proof}

\begin{Lemma}\label{lem:suz 5}\label{lem:suz 3} The $3$- and $5$-primary parts of $\H^4(\Suz)_{(5)}$ vanish.
\end{Lemma}

\begin{proof}The $5$-Sylow in $\Suz$ is contained in a maximal subgroup of shape $\rJ_2\!:\!2$. By Proposition~\ref{lem:J2}, the outer automorphism of $\rJ_2$ acts by multiplication by $-1$ on $\H^4(\rJ_2)_{(5)} =5$, and so $\H^4(\Suz)_{(5)} \subseteq \H^4(\rJ_2\!:\!2)_{(5)} = 0$.

The $3$-Sylow in $\Suz$ is contained in a maximal subgroup of shape $3^{5}\!:\!\rM_{11}$. There are two nontrivial 5-dimensional $\rM_{11}$-modules over $\bF_3$. They lead to LHS spectral sequences with $E_2$ pages:
 \begin{gather*}
 \begin{array}{ccccc}
 0 \\
 0 & 3 & 3 \\
 0 & 0 & 0 \\
 0 & 0 & 0 & 0 & 0 \\
 \bZ & 0 & 0 & 0 & 8
 \end{array}
 \qquad \text{or} \qquad
 \begin{array}{ccccc}
 0 \\
 0 & 0 & 0 \\
 0 & 3 & 0 \\
 0 & 0 & 0 & 0 & 0 \\
 \bZ & 0 & 0 & 0 & 8
 \end{array}
 \end{gather*}
The former is incompatible with $\H^3(\Suz)_{(3)} = 3$, and the latter immediate gives $\H^4(\Suz)_{(3)} \allowbreak = 0$.
\end{proof}

\begin{Lemma}\label{lem:2or4} $\H^0\big(\SWeyl (\rE_6); \H^4\big(2^{1+6}\big)\big)$ is either $\bZ/2$ or $\bZ/4$.
\end{Lemma}

We were unable to determine the exact value of $\H^0\big(\SWeyl (\rE_6); \H^4\big(2^{1+6}\big)\big)$. We remark that the order-$2$ class therein has many descriptions. It arises as $c_2$ of the unique $2^3$-dimensional complex irrep of $2^{1+6}$. It also arises as follows. By Lemma~\ref{lem:inE6}, the nonsplit extension $2^6\cdot\SWeyl (\rE_6)$ is naturally a subgroup of the compact Lie group of type $\rE_6$ (adjoint form); in this realization, $2^6$ is the group of $2$-torsion points in the maximal torus. The generator of $\H^4(B\rE_6)$ restricts to $2^6$ to the $\rE_6$ quadratic form living in $\Sym^2\big(2^6\big) \subseteq \H^4\big(2^6\big)$, and this form pulls back to $2^{1+6}$ to the $\SWeyl (\rE_6)$-fixed order-$2$ class.

\begin{proof}Let us write $J$ for $\SWeyl (\rE_6)$ and $E$ for the $6$-dimensional $\SWeyl (\rE_6)$-module over $\bF_2$. In Section~\ref{sub.extraspecial 2group} we identified $\H^4(2.E)$ as
 \begin{gather*} \big(E^*.\Alt^2(E^*).\Alt^3(E^*)/E^*\big).2. \end{gather*}
 This group has a unique nonzero element which is divisible by~$2$; it lives in the subgroup $X = E^*.\Alt^2(E^*).\Alt^3(E^*)/E^*$, and so $\H^0(J; X) \geq \bZ/2$.
 On the other hand,
 \begin{gather*} \H^0(J; E^*) = \H^0\big(J;\Alt^3(E^*)/E^*\big) = 0, \qquad \H^0\big(J;\Alt^2(E^*)\big) = \bZ/2. \end{gather*}
 Indeed, $E^*$ and $\Alt^3(E^*)/E^*$ are simple $J$-modules of dimensions $6$ and $14$ respectively, and the unique
 $J$-fixed point in $\Alt^2(E^*)$ is the underlying alternating form of the quadratic form defining the extension $2.E$. From the long exact sequence $\H^\bullet(J; A) \to \H^\bullet(J; A.B) \to \H^\bullet(J; B) \to \H^{\bullet+1}(J;A) \to \cdots$, we find
 \begin{gather*} \H^0(J;X) \leq \bZ/2\end{gather*}
 and
 \begin{gather*}
 \H^0(J; X.2) \leq (\bZ/2).(\bZ/2) = \bZ/4.\tag*{\qed}
 \end{gather*}\renewcommand{\qed}{}
\end{proof}

We may now compute the $E_2$ page of the LHS spectral sequence for the extension \linebreak $2^{1+6}.\SWeyl (\rE_6)$ using HAP and Cohomolo:
\begin{gather*} \begin{array}{ccccc}
2 \text{ or } 4 \\
0 & 2 & 0 \\
0 & 0 & 2 \\
0 & 0 & 0 & 0 & 0 \\
\bZ & 0 & 0 & 2 & 4
\end{array} \end{gather*}
From the vanishing of $E_2^{03}, E_2^{12}$, and $E_2^{21}$, we learn that the restriction map
\begin{gather*}
\H^3(\Suz;\bZ) \to \H^3(\SWeyl(\rE_6);\bZ)
\end{gather*}
is an isomorphism on $2$-primary parts. It follows that the preimage of $\SWeyl(\rE_6)$ in $2\Suz$ is the Spin double cover of $\SWeyl(\rE_6) \subset \SO(6)$, which we denote by $2\SWeyl(\rE_6)$.

The preimage of
{$2^{1+6}\cdot(\SWeyl(\rE_6)) \subseteq \Suz$}
 in $2\Suz$ is of shape $2^{1+6}\cdot(2\SWeyl(\rE_6))$. Using HAP and Cohomolo, we compute that its LHS spectral sequence has $E_2$ page:
\begin{gather*} \begin{array}{ccccc}
2 \text{ or } 4 \\
0 & 2 & 0 \\
0 & 0 & 2 \\
0 & 0 & 0 & 0 & 0 \\
\bZ & 0 & 0 & 0 & 8
\end{array} \end{gather*}

\begin{Lemma}\label{suzlemma-lifingQ8} There is a quaternion group $Q' \cong Q_8 \subseteq 2^{1+6}.2\SWeyl (\rE_6)$ such that the central element of $Q'$ maps to an element of $\SWeyl(\rE_6)$ of conjugacy class~$2\ra$.
\end{Lemma}

There are two conjugacy classes of elements of order $2$ in $\SWeyl (\rE_6)$. They can be distinguished by the orders of their preimages in $2\SWeyl (\rE_6)$: elements in class $2\ra$ lift with order~$2$ (and both lifts are conjugate in $2\SWeyl(\rE_6)$), whereas elements in class $2\rb$ lift with order~$4$. So the content of the Lemma is the existence of such a $Q'$ such that the composition $Q' \mono 2^{1+6}.2\SWeyl (\rE_6) \to \SWeyl (\rE_6)$ is injective.

\begin{proof} We will find this $Q'$ by finding a larger group: we will hunt for a binary tetrahedral group $2A_4 \subseteq 2^{1+6}.2\SWeyl (\rE_6)$, and then take $Q'$ to be its $2$-Sylow. Let us say that copy of~$2A_4$ inside some extension of $\SWeyl(\rE_6)$ is ``{appropriate}'' if the central element of that $2A_4$ maps to class $2\ra \in \SWeyl (\rE_6)$. Then for our search, it suffices to find an appropriate $2A_4 \subseteq 2^6.\SWeyl(\rE_6)$. Indeed, $\H_1(2A_4) = 3$ and $\H_2(2A_4) = 0$, and so any $2A_4 \subseteq 2^6.\SWeyl(\rE_6)$ will lift to a $2A_4 \subseteq 2^{1+6}.2\SWeyl (\rE_6)$.

To find an appropriate $2A_4 \subseteq 2^6.\SWeyl(\rE_6)$, we recognize that
 \begin{gather*} 2^6.\SWeyl(\rE_6) \subseteq 2^6.\mathrm{Weyl}(\rE_6) \subseteq (\text{maximal torus}).\mathrm{Weyl}(\rE_6) \subseteq \text{Lie group }\rE_6, \end{gather*}
 where the group $2^6$ is nothing but the $2$-torsion points in the maximal torus. Consider the standard Lie group embedding $\rF_4 \subseteq \rE_6$. This leads to an embedding
 \begin{gather*} 2^4.\SWeyl(\rF_4) \subseteq 2^6.\mathrm{SWeyl}(\rE_6). \end{gather*}
 covering an embedding $\SWeyl(\rF_4) \subseteq \mathrm{SWeyl}(\rE_6)$. Because the Lie group $\rF_4$ has no outer automorphisms, the Weyl group of $\rF_4$ contains all automorphisms of the $\rF_4$ root lattice (isomorphic to the $D_4$ lattice). There is a standard identification between the $\rF_4$ lattice and the Hurwitz quaternions $\big\{a + bi + cj + dk \in \bH \,|\, a,b,c,d \in \bZ\big\} \sqcup \big\{a + bi + cj + dk \in \bH \,|\, a,b,c,d \in \bZ+\frac12\big\}$. The group of units in the Hurwitz numbers is a copy of~$2A_4$. This provides a subgroup $2A_4 \subseteq \SWeyl(\rF_4) \subseteq \SWeyl(\rE_6)$, which is easily seen to be appropriate: central $2 \subseteq 2A_4$ acts by $-1$ on the $\rF_4$ lattice, and so with trace~$-2$ on the~$\rE_6$ lattice, and so fuses to class $2\ra \in \SWeyl(\rE_6)$.

 Finally, we claim that the extension $2^4.2A_4$ splits. Indeed, the action of $2A_4$ on $2^4$ is the mod-2 reduction of the action on the $\rF_4$ lattice, and for this action, $\H^2\big(2A_4; 2^4\big) = 0$. Thus we have found an appropriate $2A_4 \subseteq 2^4.\SWeyl(\rF_4) \subseteq 2^6.\mathrm{SWeyl}(\rE_6)$.
\end{proof}

\begin{Lemma} \label{suz trivial intersection} The pullbacks \[
 \begin{tikzcd}
 \H^4(2\Suz)_{(2)} \arrow[r,hook] &
 \H^4\big(2^{1+6}.2\SWeyl (\rE_6)\big)
 & \H^4\big(2\SWeyl (\rE_6)\big)_{(2)} \arrow[l,hook']
 \end{tikzcd}
 \] have trivial intersection.
\end{Lemma}

\begin{proof} Let $Q' \subseteq 2^{1+6}.2\SWeyl (\rE_6)$ denote the quaternion group found in Lemma~\ref{suzlemma-lifingQ8}, and let $Q \subseteq 2\SWeyl(\rE_6)$ denote it image. Then $Q \cong Q_8$ is another quaternion group since the center of~$Q'$ is not in the kernel of the map $Q' \to Q$.

 We claim that $\H^4(2\SWeyl(\rE_6)) \to \H^4(Q)$ is an isomorphism. Indeed, consider either 4-dimensional faithful complex representations of $2\SWeyl(\rE_6)$. Class $2\ra$ acts on this representation with trace $0$. It follows that this representation decomposes over $Q$ as one copy of the $2$-dimensional irrep plus two copies of the same $1$-dimensional irrep, and so $c_2(\text{4-dim rep})|_Q$ has order $8$. We furthermore learn that $c_2(\text{4-dim rep})$ generates $\H^4(2\SWeyl(\rE_6))$.

 We henceforth write $\alpha \in \H^4(2\SWeyl (\rE_6)) \cong \bZ/8$ for this generator. Let $\tilde\alpha$ denote its pullback to $2^{1+6}.2\SWeyl (\rE_6)$.
 To prove the proposition, it suffices to show that $4\tilde\alpha$ is not in the image of $\H^4(2\Suz)$.

The central element of $Q'$ is an order-$2$ lift of class $2\ra \in \SWeyl(\rE_6)$. Any such lift fuses to class $2\ra \in \Suz$. But $2^{1+6}.\SWeyl(\rE_6)$ is the centralizer of an element of class $2\ra \in \Suz$. It follows that $Q'$ is conjugate in $2\Suz$ to some other quaternion group $Q'' \subseteq 2^{1+6}.2\SWeyl(\rE_6)$ whose central element
covers the center of $2\SWeyl(\rE_6)$.

Since $Q'$ is a lift of $Q$, we find that $\tilde\alpha|_{Q'}$ is a generator of $\H^4(Q')$, and so $4\tilde\alpha|_{Q'} \neq 0$.
On the other hand, since the center of $Q''$ maps to something central in $2\SWeyl$, the $4$-dimensional representation of $2\SWeyl$ breaks up over the image of $Q''$ as either four 1-dimensional representations or two copies of the 2-dimensional representation, and in either case we find that $\tilde\alpha|_{Q''} = c_2(\text{4-dim rep})|_{Q''}$ has order at most $4$, so that $4\tilde\alpha|_{Q''} = 0$. Since $Q'$ and $Q''$ are conjugate in $2\Suz$, the class $4\tilde\alpha$ cannot be the restriction of a class in $\H^4(2\Suz)$.
\end{proof}

\begin{proof}[Proof of Theorem~\ref{thm:suz}] $\H^4(\Suz)_{(p)}$ vanishes for $p\geq 7$ by Lemma~\ref{large primes}, {and for $p=5$ by Lemma~\ref{lem:suz 5}}. Lemma~\ref{lem:suz 3} gave $\H^4(\Suz)_{(3)} = \H^4(2\Suz)_{(3)} = 0$, and so Lemma~\ref{lemma:schur cover} provides the upper bound $\H^4(3\Suz)_{(3)} \leq 3$. But Lemma~\ref{lem:suz lower bound} provides the lower bound $\H^4(3\Suz)_{(3)} \geq 3$, and so{\samepage
 \begin{gather*} \H^4(3\Suz)_{(3)} \cong \H^4(6\Suz)_{(3)} \cong \bZ/3, \end{gather*}
 generated by the $3$-part of $c_2(V)$.}

We now argue that $\H^4(2\Suz) = 8$. Lemma~\ref{lem:suz lower bound} implies that $\H^4(2\Suz)$ contains an element of order $8$, namely the $2$-part of $c_2(V)$, where $V$ denotes the 12-dimensional irrep of $6\Suz$. Lemma~\ref{suz trivial intersection} implies that $\H^4(2\Suz)$ has order at most $16$. Furthermore, since the $2$-part of~$c_2(V)$ has order~$8$, its restriction to~$2^{1+6}$ must be nonzero. On the other hand, for every representation~$W$ of~$2^{1+6}$, $c_2(W) \in \H^4\big(2^{1+6}\big)$ has order dividing~$2$. (Indeed, for the one-dimensional irreps of $2^{1+6}$ this is automatic, and for the unique irrep of dimension $2^3$ it is a straightforward computation.) Thus $c_2(V)$ restricts to the unique class in $\H^4(2^{1+6})$ which is divisible by $2$. From this we learn that the only way for $\H^4(2\Suz)$ to have order $16$ is if $c_2(V)$ is divisible by~$2$.

 Suppose that there were a class ``$\frac12 c_2(V)$''. Since $c_2(V)$ restricts to $2D_8$ with order~$8$, this class $\frac12 c_2(V)$ would have to have order $16$ when restricted to $2D_8$. The order-16 classes in $\H^4(2D_8)$ are the ones that have nontrivial restriction to the center of $2D_8$, which by construction is the center of $2\Suz$. But all classes in $\H^4(2^{1+6}.2\SWeyl(\rE_6))$, hence all classes in $\H^4(2\Suz)$, restrict trivially to the center. This proves
 \begin{gather*} \H^4(2\Suz) \cong \bZ/8, \end{gather*}
 generated by the $2$-part of $c_2(V)$.

 Finally, we argue that $\H^4(\Suz)_{(2)} \cong \H^4(3\Suz)_{(2)} \cong \bZ/4$ by repeating the logic from \cite[Theorem 5.3]{jft}. Namely, we have a commutative diagram
 \[
 \begin{tikzcd}
 2D_8 \arrow[r,hook] \arrow[d,twoheadrightarrow] & 6\Suz \arrow[r,hook] \arrow[d,twoheadrightarrow] & 2\Co_1 \arrow[d,twoheadrightarrow] \\
 D_8 \arrow[r,hook] & 3\Suz \arrow[r,hook] & \Co_1,
 \end{tikzcd}
 \]
 which, upon applying $\H^4(-)_{(2)}$, gives the diagram
 \[ \begin{tikzcd}
 \bZ/16 & \bZ/8 \arrow[l,hook'] & \bZ/8 \arrow[l,hook'] \\
 \bZ/4 \times (\bZ/2)^2 \arrow[u] & \H^4(3\Suz)_{(2)} \arrow[u,hook'] \arrow[l] & \bZ/4. \arrow[u,hook'] \arrow[l]
 \end{tikzcd} \]
 The north-then-west compositions are injective, and so both southern arrows must be injective. It follows that
\begin{gather*} \H^4(\Suz) \cong \H^4(3\Suz)_{(2)} \cong \bZ/4.\tag*{\qed}
\end{gather*}\renewcommand{\qed}{}
\end{proof}

\section{Pariahs} \label{sec:pariahs}

\subsection{Janko groups 1 and 3}

Using the permutation representations listed in the ATLAS, HAP is able to compute
\begin{gather*} \H_3(\rJ_1) \cong \bZ/30, \qquad \H_3(\rJ_3) \cong \bZ/15. \end{gather*}
HAP is unable to compute $\H_3(3\rJ_3)$ directly.
\begin{Theorem}
 $\H_3(3\rJ_3) \cong (\bZ/3)^2 \times \bZ/5$. Both $\H^4(\rJ_3)$ and $\H^4(3\rJ_3)$ consist entirely of Chern classes, and are detected on cyclic subgroups.
\end{Theorem}
\begin{proof}
 Let $V_{323}$ denote a choice of complex irrep of $\rJ_3$ of dimension $323$. (The two choices are the characters $\chi_4$ and $\chi_5$ when listed by increasing dimension; both are real, and are exchanged by the outer automorphism of $\rJ_3$.) Then $c_2(V_{323})$ restricts nontrivially to all cyclic subgroups of order $5$ in $\rJ_3$. For the $3$-parts of $\H^4(\rJ_3)$, we focus on the conjugacy classes $3\ra$ and $9\ra$, and any choice of $1920$-dimensional irrep $V_{1920}$ (there are three such irreps, all real, with characters~$\chi_{14}$,~$\chi_{15}$, and~$\chi_{16}$). Then $c_2(V_{1920})$ restricts nontrivially to both $\langle 3\ra\rangle$ and $\langle 9\ra \rangle$.

 Choose any lift of $3\ra \in \rJ_3$ to $3\rJ_3$, for example $3\rc \in 3\rJ_3$. (The classes $3\ra, 3\rb \in 3\rJ_3$ are central.) The class $9\ra \in \rJ_3$ lifts to a single conjugacy class in $3\rJ_3$, also called $9\ra$. With these new names, we still have that $c_2(V_{1920})|_{3\rc}$ and $c_2(V_{1920})|_{9\ra}$ are nontrivial.
 Finally, consider the smallest faithful representation $V_{18}$ of $3\rJ_3$, with dimension $18$ and character $\chi_{22}$. Then $c_2(V_{18})|_{3\rc} = 0$, but $c_2(V_{18})$ restricts with order $3$ to $\langle 9\ra\rangle$.

 It follows that the image of the map $\H^4(3\rJ_3)_{(3)} \to \H^4(\langle 3\rc\rangle) \times \H^4(\langle 9\ra\rangle)$ is not cyclic. On the other hand, the HAP computation of $\H_3(\rJ_3)_{(3)}$ together with Lemma~\ref{lemma:schur cover} imply that the domain has order at most $9$. So $\H^4(3\rJ_3)_{(3)} \cong (\bZ/3)^2$.
\end{proof}

\subsection{O'Nan group}

\begin{Theorem}\label{thm:on}
 $\H_3(\mathrm{O'N}) \cong \H_3(3\mathrm{O'N}) \cong \bZ/8$.
\end{Theorem}

\begin{proof} The $p$-parts of $\H^4(\mathrm{O'N})$ vanish for $p=5$ and $p\geq 11$ by Lemma~\ref{large primes}. The $7$-Sylow is contained in a subgroup isomorphic to $\mathrm{PSL}_3(7)$, and a HAP computation gives $\H^4(\mathrm{PSL}_3(7)) \cong \bZ/16$.

 The 3-Sylow in $3\mathrm{O'N}$ is an extraspecial group of shape $3^{1+4}$, and its normalizer $N$ is a maximal subgroup of shape $N = 3^{1+4}:2^{1+4}.D_{10}$. HAP computes
 \begin{gather*} \H_3(N) \cong \bZ/4 \times \bZ/8 \times \bZ/5.\end{gather*}
 It follows that $\H^4(3\mathrm{O'N})_{(3)}$, and hence also $\H^4(\mathrm{O'N})_{(3)}$, vanishes.

 The $2$-Sylow in $\mathrm{O'N}$ is contained inside the nonsplit extension $4^3\cdot \mathrm{GL}_3(2)$. According to Lemma~\ref{lemma:Dempwolff}(4), \begin{gather*}\H^4\big(4^3\cdot \mathrm{GL}_3(2)\big)_{(2)} \cong (\bZ/2)^2 \times \bZ/8.\end{gather*} The $\bF_2$-cohomology ring of~$\mathrm{O'N}$, including the action of Steenrod squares, was computed by \cite{MR1345306}. They find that
 \begin{gather*}\H^1(\mathrm{O'N};\bF_2) = \H^2(\mathrm{O'N};\bF_2) = 0, \qquad
 \H^3(\mathrm{O'N};\bF_2) \cong \H^4(\mathrm{O'N};\bF_2) \cong \bF_2,\end{gather*}
 but \begin{gather*}\Sq^1 = 0\colon \H^3(\mathrm{O'N};\bF_2) \to \H^4(\mathrm{O'N};\bF_2).\end{gather*} It follows that $\H^4(\mathrm{O'N})_{(2)}$ is cyclic of order strictly greater than $2$. Since $\H^4(\mathrm{O'N})_{(2)}$ is a direct summand of $\H^4\big(4^3\cdot \mathrm{GL}_3(2)\big)_{(2)}$, the only option is $\H^4(\mathrm{O'N})_{(2)} \cong \bZ/8$.
\end{proof}

\subsection{Janko group 4 and Lyons group}

The two largest Pariahs are Janko's fourth group $\rJ_4$ and Lyons' group $\mathrm{Ly}$. Both have trivial Schur multiplier~\cite{MR2611672}, and so their cohomologies vanish in degrees $\leq 3$. We find that in fact their cohomologies vanish in degrees $\leq 4$. Only one other sporadic group has this property: the cohomology of $\rM_{23}$ vanishes in degrees $\leq 5$~\cite{MR1736514}.

\begin{Theorem} $\H^4(\rJ_4)$ is trivial.
\end{Theorem}

The full cohomology ring of $\rJ_4$ is computed away from the prime $2$ in \cite{MR1198410}.

\begin{proof} The only primes not covered by Lemma~\ref{large primes} are $2$, $3$, and $11$.
 The $11$-Sylow in $\rJ_4$ is contained in a maximal subgroup of shape $11^{1+2}:(5\times 2S_4)$. It is easy to check that $\H^4\big(11^{1+2}:(5\times 2S_4)\big)_{(11)} = 0$; for example by observing that the central $10 \subseteq 5\times 2S_4$ acts on $11^2$ through the isomorphism $10 \cong (\bZ/11)^\times$ and applying Lemma~\ref{lemma:central character}.

 The $3$-Sylow is contained in a maximal subgroup of shape $2^{11}:\rM_{24}$. There are two conjugacy classes of elements of order $3$ in $\rM_{24}$; the restriction map $\H^4(\rM_{24})_{(3)} \cong \bZ/3 \to \H^4(\langle 3\ra\rangle)$ is zero, whereas $\H^4(\rM_{24})_{(3)} \cong \bZ/3 \to \H^4(\langle 3\rb\rangle)$ is an isomorphism \cite{GPRV}. But $\rJ_4$ has only one conjugacy class of order $3$. It follows that $\H_4(\rJ_4)_{(3)} = 0$.

 The $2$-Sylow is contained in $2^{11}:\rM_{24}$, and also in a maximal subgroup of shape $2^{1+22}.3\rM_{22}.2$ centralizing conjugacy class $2\ra \in \rJ_4$. This latter subgroup turns out to be more useful. Using Cohomolo and Proposition~\ref{prop:mathieu}, we find that the $E_2$ page for the LHS spectral sequence reads
 \begin{gather*} \begin{array}{ccccc}
 \leq 4 \\
 0 & 0 & 2 \\
 0 & 0 & 0 \\
 0 & 0 & 0 & 0 & 0 \\
 \bZ & 0 & 2 & 4 & 2^2 \times 3
 \end{array} \end{gather*}
 The entry ``$\leq 4$'' in bidegree $(0,4)$ comes from the LES for the extension
 \begin{gather*} \H^4\big(2^{1+12}\big) = 2^{12}.\Alt^2\big(2^{12}\big).\big(\Alt^3\big(2^{12}\big)/2^{12}\big).2 \end{gather*}
 from Section~\ref{sub.extraspecial 2group} and the calculations
 \begin{gather*} \H^0\big(3\rM_{22}.2; 2^{12}\big) = \H^1\big(3\rM_{22}.2; 2^{12}\big) = \H^0\big(3\rM_{22}.2; \Alt^3\big(2^{12}\big)\big) = 0, \\ \H^0\big(3\rM_{22}.2; \Alt^2\big(2^{12}\big)\big) = 2. \end{gather*}

 We showed during the proof of Theorem~\ref{thm.co2} that $\H^4(\rM_{22}:2)_{(2)} = 2^2$ is detected by restricting to the three conjugacy classes of order $2$ in $\rM_{22}:2$. All three of these classes have order-$2$ preimages in $2^{1+22}.3\rM_{22}.2$. But both conjugacy classes of order $2$ in $\rJ_4$ meet $2^{1+22}$. It follows that the images of $\H^4(\rM_{22}:2)_{(2)} \to \H^4\big(2^{1+22}.3\rM_{22}.2\big)_{(2)}$ and $\H^4(\rJ_4)_{(2)} \to \H^4\big(2^{1+22}.3\rM_{22}.2\big)_{(2)}$ have trivial intersection.

In particular, if $\H^4(\rJ_4) \neq 0$, then it contains an order-$2$ class $\alpha$ whose restriction to $2^{1+12}$ is the unique element $\tilde q \in \H^4\big(2^{1+12}\big)$ which is twice some other element. That unique element is pulled back from $\H^4(2^{12})$, where it corresponds to the quadratic form $q \in \Sym^2(2^{12})$ defining the extension $2^{1+12}$. Choose any pair of vectors $v_1,v_2 \in 2^{12}$ such that $q(v_1) = q(v_2) \neq 0$. Their lifts generate a quaternion group $Q_8 = 2^{1+2} \subseteq 2^{1+12}$, and $\tilde q \in \H^4(2^{1+12})$ restricts nontrivially to that quaternion group. (These lifts of $v_1$, $v_2$ have order $4$ in $2^{1+12}$, and so are in conjugacy class $4\ra$ in $2^{1+12}.3\rM_{22}.2$.)

Choose $\tilde g \in 2^{1+12}.3\rM_{22}.2$ in conjugacy class $4\rb$. The character table libraries confirm that this~$\tilde g$ has the following properties: $\tilde g^2$ is the nontrivial central element in $2^{1+12} \subseteq 2^{1+12}.3\rM_{12}.2$; the image $g$ in $3\rM_{22}.2$ of $\tilde g$ is in conjugacy class~$2\ra$. In particular, $g$ acts on $2^{12}$ fixing $8$ dimensions. Regardless of the Arf invariant of $q$, one can find a basis such that $q$ vanishes on at most one basis vector; thus we can find a vector $v_1 \in 2^{12}$ with $q(v) \neq 0$ and $vg \neq v$. (Following GAP's conventions, we write the action of $3\rM_{22}.2$ on $2^{12}$ from the right.) Set $v_2 = vg$; then $q(v_2) \neq 0$, and so the lifts of $v_1,v_2$ generate a $Q_8$ as in the previous paragraph.

Furthermore, we have arranged for the lifts of $v_1$, $v_2$ together with $\tilde g$ to generate a binary dihedral group $2D_8 \subseteq 2^{1+12}.3\rM_{22}.2$ extending this~$Q_8$. Suppose that $\H^4(\rJ_4) \neq 0$, and let $\alpha$ denote its order-$2$ class. Then $\alpha|_{2D_8}$ has order $2$, and so is divisible by $8$. But then $\alpha|_{Q_8} = 0$. This contradicts the fact that~$Q_8$ detected $\tilde q$, and so $\H^4(\rJ_4)$ must vanish.
 \end{proof}

\begin{Theorem}$\H^4(\mathrm{Ly})$ is trivial.
\end{Theorem}

\begin{proof}For $p>5$, the Sylow $p$-subgroup of $\mathrm{Ly}$ is cyclic, so Lemma~\ref{large primes} applies.

$\mathrm{G}_2(5)$ and $3\mathrm{McL}:2$ are subgroups (in fact, maximal subgroups) of the $\mathrm{Ly}$ \cite[Propositions~2.5 and~5.4]{MR0299674}. The $5$-Sylow is contained in $\mathrm{G}_2(5)$, so the $\H^4(\mathrm{Ly})_{(5)}$ vanishes by Lemma~\ref{lemma:g2q}.

The $2$- and $3$-Sylows are contained in $3\mathrm{McL}:2$. By Theorem~\ref{thm.mcl}, the only cohomology of the latter is pulled back from the quotient $\bZ/2$, and so is detected on a conjugacy class of order~$2$ (specifically, class $2\rb \in 3\mathrm{McL}:2$). But $\mathrm{Ly}$ has only one conjugacy class of order~$2$, and it meets $3\mathrm{McL} \subseteq 3\mathrm{McL}:2$ (since $3\mathrm{McL}$ has a class of order~$2$). It follows that the pullback $\H^4(2) \overset\sim\to \H^4(3\mathrm{McL}:2)$ and the restriction $\H^4(\mathrm{Ly})_{(2)} \to \H^4(3\mathrm{McL}:2)$ have nonintersecting images.
\end{proof}

\section{Monster sections}\label{sec:monsters}

\subsection{Held group}

The Held group is small enough to be accessible by the methods of Sections~\ref{sec:methods}--\ref{subsec:charclasses}.

\begin{Theorem}\label{thm:he}
 $\H^4(\mathrm{He}) \cong \bZ/12$. It is spanned by fractional Pontryagin classes.
\end{Theorem}
\begin{proof}
 The primes not covered by Lemma~\ref{large primes} are $p=2$, $3$, and $7$.

 The normalizer of a $7$-Sylow in $\mathrm{He}$ has shape $7^{1+2}:(3 \times S_3)$. There are no $7$s in its low cohomology: $\H^4\big(7^{1+2}:(3 \times S_3)\big) = \H^4(3 \times S_3) = 2 \times 3^2$.

 The $3$-Sylow in $\mathrm{He}$ is inside a maximal subgroup of shape $3S_7$, with $\H^4(3S_7) = 2^2 \times 4 \times 3^2$. We claim that the inclusions $\H^4(\mathrm{He})_{(3)} \to \H^4(3S_7)_{(3)}$ and $\H^4(S_7)_{(3)} = 3 \to \H^4(3S_7)_{(3)}$ have nonintersecting images, giving an upper bound $\H^4(\mathrm{He})_{(3)} \leq 3$. To show this, we first observe that if $C_3$ is a cyclic group of order $3$ then the two nonzero classes in $\H^4(C_3) \cong \bZ/3$ are canonically distinguished: one, which we will call $t^2 \in \H^4(C_3)$, is the cup-square of both nonzero classes in $\H^2(C_3) \cong \bZ/3$, and the other is not a cup-square. Now consider the class $c_2(\Perm) \in \H^4(S_7)$, where $\Perm$ denotes the defining permutation representation. There are three conjugacy classes of order $3$ in $3S_7$: the ``central'' one (not actually central~-- it is inverted by the odd elements of~$S_7$), and two that act on $\Perm$ with cycle structures $1^4 3^1$ and $1^1 3^2$, respectively. It follows that $c_2(\Perm)|_{\langle 1^4 3^1\rangle} = -t^2$ whereas $c_2(\Perm)|_{\langle 1^1 3^2\rangle} = +t^2$, meaning that $c_2(\Perm)$ takes different values on these two classes. However, these two classes merge to the same class in $\mathrm{He}$, and so $c_2(\Perm)$ cannot be the restriction of a cohomology class on $\mathrm{He}$.

 To establish the lower bound $\H^4(\mathrm{He})_{(3)} \geq 3$, we observe that the smallest irrep of $\mathrm{He}$ has dimension $51$, and conjugacy class $3\rb \in \mathrm{He}$ acts with trace $0$, and so $c_2$ of this irrep, when restricted to $\langle 3\rb\rangle$, does not vanish.

 For the prime $p=2$, we use the $2$-Sylow-containing maximal subgroup of shape $2^6:3S_6$. The $E_2$ page (localized at $p=2$) of the LHS spectral sequence for this extension reads
 \begin{gather*} \begin{array}{ccccc}
 E_2^{04} \\
 0 & 2 & 2^2 \\
 0 & 0 & 0 \\
 0 & 0 & 0 & 0 & 0 \\
 \bZ & 0 & 2 & 2 & 2^2\times 4
 \end{array} \end{gather*}
 with
 \begin{gather*} E_2^{04} = \H^0\big(3S_6; V . \Alt^2(V) . \Alt^3(V)\big), \qquad V = \big(2^6\big)^*. \end{gather*}
 Since
 \begin{gather*} \H^0(3S_6; V) = \H^0\big(3S_6; \Alt^2(V)\big) = 0, \qquad \H^0\big(3S_6; \Alt^3(V)\big) \cong \bZ/2, \end{gather*}
 we find that
 \begin{gather*} E_2^{04} \leq 2.\end{gather*}

 We claim that the inclusions $\H^4(S_6)_{(2)} \cong \H^4(3S_6)_{(2)} \to \H^4\big(2^6:3S_6\big)$ and $\H^4(\mathrm{He})_{(2)} \to \H^4\big(2^6:3S_6\big)$ have trivial intersection. To establish this claim, we first study $\H^4(S_6)_{(2)}$. There are four 5-dimensional complex irreps of $S_6$, corresponding to the characters $\chi_3$, $\chi_4$, $\chi_5$, and $\chi_6$. Their second Chern classes, when restricted to the conjugacy classes $2\rb$, $4\ra$, and $4\rb$ in $S_6$, are
 \begin{gather*} \begin{array}{c|ccc}
 & 2\rb & 4\ra & 4\rb \\ \hline &&& \\[-10pt]
 c_2(\chi_3) & 0 & -t^2 & +t^2 \\
 c_2(\chi_4) & 0 & -t^2 & -t^2 \\
 c_2(\chi_5) & 1 & -t^2 & -t^2 \\
 c_2(\chi_6) & 1 & -t^2 & +t^2
 \end{array} \end{gather*}
 Our notation is the following. The two classes in $\H^4(\langle 2\rb\rangle) \cong \bZ/2$ are ``$0$'' and ``$1$''. If $C_4$ is a~cyclic group of order $4$, the two generators of $\H^2(C_4)$ have the same cup-square in $\H^4(C_4)$, which we call ``$+t^2$''; the generator of $\H^4(C_4)$ which is not a cup-square is called~``$-t^2$''.

 The images of $\{c_2(\chi_3), c_2(\chi_4), c_2(\chi_5), c_2(\chi_6)\}$ in $\H^4(\langle 2\rb\rangle) \times \H^4(\langle 4\ra\rangle) \times \H^4(\langle 4\rb\rangle) \cong \bZ/2 \times (\bZ/4)^2$ together span a group isomorphic to $(\bZ/2)^2 \times \bZ/4$. It follows that these four classes span $\H^4(S_6)_{(2)} \cong (\bZ/2)^2 \times \bZ/4$ and that the restriction map $\H^4(S_6)_{(2)} \to \H^4(\langle 2\rb\rangle) \times \H^4(\langle 4\ra\rangle) \times \H^4(\langle 4\rb\rangle)$ is an injection.

 On the other hand, the subgroup $2^6 \subseteq 2^6:3S_6 \subseteq \mathrm{He}$ meets both conjugacy classes of order~$2$, and the order-$4$ preimages in $2^6 \subseteq 2^6:3S_6$ of the classes $4\ra,4\rb \in S_6$ are $\mathrm{He}$-conjugate to preimages of order-$2$ elements in~$S_6$. It follows that the image of $\H^4(S_6)_{(2)} \cong \H^4(3S_6)_{(2)} \to \H^4\big(2^6:3S_6\big)$ does not intersect $\H^4(\mathrm{He})$.

 We have so far established the following upper bound on $\H^4(\mathrm{He})_{(2)}$: it is a group of order at most $4$.
 The last ingredient needed is a cohomology class of order divisible by $4$. Let $V$ denote the irreducible $\mathrm{He}$-module with character $\chi_{19}$: it is real and of dimension $7650$. Consider the conjugacy class $4\ra \in \mathrm{He}$. It squares to $2\ra$, and
 \begin{gather*} \chi_{19}(4\ra) = 6, \qquad \chi_{19}(2\ra) = 90. \end{gather*}
 Therefore $4\ra$ acts with spectrum $1^{1938} i^{1890} (-1)^{1932} (-i)^{1890}$, and so the total Chern class of $V$, when restricted to $\langle 4\ra\rangle$, is
 \begin{gather*} c(V) = 1^{1938} (1-t)^{1890} (1-2t)^{1932} (1+t)^{1890} = 1 - 2 t^2 + \cdots \pmod {4t}. \end{gather*}
 In particular, $c_2(V)|_{\langle 4\ra\rangle} \neq 0$. But $V$ is a real representation, and therefore Spin (since $\mathrm{He}$ has trivial Schur multiplier). So it has a fractional Pontryagin class, twice of which is~$c_2(V)$. It follows that $\frac{p_1}2(V)$ has order divisible by~$4$, and so $\H^4(\mathrm{He})_{(2)} \cong \bZ/4$.
\end{proof}

\subsection{Harada--Norton and Thompson groups}

We were able to obtain partial results about the Harada--Norton and Thompson groups $\mathrm{HN}$ and~$\mathrm{Th}$.

\begin{Theorem}
 $\H^4\big(\mathrm{HN}; \bZ\big[\frac12\big]\big) \cong \bZ/3$. At the prime~$2$, $\H^4(\mathrm{HN})_{(2)}$ has order at most $16$ and exponent at most~$8$.
\end{Theorem}

\begin{proof} Lemma~\ref{large primes} handles the primes $p\geq 7$. The $5$-Sylow in $\mathrm{HN}$ is contained in a maximal subgroup of shape $5^{1+4}.2^{1+4}.5.4$; the LHS spectral sequence gives $\H^4\big(5^{1+4}.2^{1+4}.5.4\big)_{(5)} = 0$.

 The $3$-Sylow is inside a maximal subgroup of shape $3^{1+4}:4A_5$, where by ``$4A_5$'' we mean the ``diagonal'' central extension $2.(A_5 \times 2)$. HAP can directly compute
 \begin{gather*} \H^4\big(3^{1+4}:4A_5\big)_{(3)} = 3^2.\end{gather*}
 We claim that the generator of $\H^4(4A_5)_{(3)} = 3$, when pulled back to $3^{1+4}:4A_5$, is not the restriction of a class on $\mathrm{HN}$. Indeed, that generator has nontrivial restriction to the elements of order~$3$ in~$4A_5$, and so its pullback has nonzero restriction to some elements of order $3$. But both conjugacy classes of order $3$ in $\mathrm{HN}$ meet $3^{1+4} \subseteq 3^{1+4}:4A_5$. Finally, we claim that $\H^4(\mathrm{HN})_{(3)} \neq 0$. There is a unique conjugacy class of order $9$ in $\mathrm{HN}$, and its traces on either $133$-dimensional representation, which characters $\chi_2$ and $\chi_3$, are $\chi_2(9\ra) = 1$ and $\chi_2\big(9\ra^3\big) = \chi_2(3\rb) = -2$. From this one can compute that $c_2(\chi_2)|_{\langle 9\ra\rangle} \neq 0$.

 The $2$-Sylow in $\mathrm{HN}$ is contained in the centralizer of conjugacy class $2\rb$, which has shape $2^{1+8}_+.(A_5 \wr 2)$. The $E_2$ page of the corresponding LHS spectral sequence provides an upper bound of $\big|\H^4(\mathrm{HN})_{(2)}\big| \leq 2^6$. Let $E = 2^8 \cong \big(2^8\big)^*$ and $J = (A_5 \wr 2)$; then
 \begin{gather*} \H^0(J;E) = \H^1(J;E) = \H^0\big(J; \Alt^3(E)\big) = 0, \qquad \H^2(J;E) \cong \bZ/2,\end{gather*}
 while
 \begin{gather*} \H^0\big(J;\Alt^2(E)\big) \cong \bZ/2, \qquad \H^0\big(J;\Alt^2(E)/2\big) \cong \bZ/2; \qquad \H^1\big(J;\Alt^2(E)/2\big) \cong (\bZ/2)^2. \end{gather*}
 Therefore the $E_2$ page of the LHS spectral sequence, after localizing at $2$, reads
 \begin{gather*} \begin{array}{ccccc}
 2 \\
 2 & 2^2 \\
 0 & 0 & 2 \\
 0 & 0 & 0 & 0 & 0\\
 \bZ & 0 & 2 & 2 & 2^2
 \end{array} \end{gather*}
 As usual, the image of $\H^4(A_5 \wr 2)_{(2)} \to \H^4\big(2^{1+8}_+.(A_5 \wr 2)\big)_{(2)}$ does not intersect $\H^4(\mathrm{HN})$: the former is detected on elements of order $2$, but both conjugacy classes of order $2$ in $\mathrm{HN}$ meet $2^{1+8}_+ \subseteq 2^{1+8}_+.(A_5 \wr 2)$.
\end{proof}

\begin{Theorem} $\H^4\big(\mathrm{Th}; \bZ\big[\frac13\big]\big) \cong \bZ/8$.
\end{Theorem}

\begin{proof}Lemma~\ref{large primes} handles the primes $p\geq 7$. The $5$-Sylow is inside $5^{1+5}:4S_4$, and Lemma~\ref{lemma:central character} implies $\H^4\big(5^{1+5}:4S_4\big)_{(5)} = 0$. The $3$-Sylow does not live in any nice maximal subgroups, and so we cannot compute $\H^4(\mathrm{Th})_{(3)}$.

The $2$-Sylow in $\mathrm{Th}$ is contained in the Dempwolff group of shape $2^5\cdot \GL_5(2)$. According to Lemma~\ref{lemma:Dempwolff}(3), $\H^4\big(2^5\cdot \GL_5(2)\big)_{(2)} = 8$, and the proof of that lemma established that $c_2(V) \neq 0$, where~$V$ denotes the $248$-dimensional irrep of $2^5\cdot \GL_5(2)$. This irrep $V$ extends to the defining $248$-dimensional irrep of~$\mathrm{Th}$, and so the restriction map $\H^4(\mathrm{Th})_{(2)} \to \H^4\big(2^5\cdot \GL_5(2)\big)_{(2)}$ is nonzero. Since the image of that restriction map is a direct summand, we must have $\H^4(\mathrm{Th})_{(2)} \cong \H^4\big(2^5\cdot \GL_5(2)\big)_{(2)} \cong \bZ/8$.
\end{proof}

\subsection{Fischer groups}

The Fischer groups $\Fi_{22}$, $\Fi_{23}$, and $\Fi_{24}$ are a ``third generation'' version of the Mathieu groups $\rM_{22}$, $\rM_{23}$, and $\rM_{24}$. Specifically, the $2$-Sylow in $\Fi_N$, for $N \in \{22,23,24\}$, lives in a subgroup of shape $2^{\lceil N/2\rceil - 1}.\rM_N$, making the calculation of $\H^4(\Fi_N)_{(2)}$ systematic. The extension $2^{\lceil N/2\rceil - 1}.\rM_N$ splits for $N = 22$ and does not split for $N = 23$ and~$24$.

The $3$-Sylows in all cases are contained in orthogonal groups over $\bF_3$. We were able to handle the $3$-parts of $\H^4(G)$ for $G = \Fi_{22}$ and $3\Fi_{22}$, but not for the larger Fischer groups.

\begin{Theorem}\label{thm.fi} The Fischer groups have the following cohomologies away from the prime $3$:
 \begin{enumerate}\itemsep=0pt
 \item[$1)$] $\H^4\big(\Fi_{22}; \bZ\big[\frac13\big]\big) = 0$; 
 \item[$2)$] $\H^4\big(2\Fi_{22}; \bZ\big[\frac13\big]\big)$ has order $2$ or $4$; 
 \item[$3)$] $\H^4\big(\Fi_{23};\bZ\big[\frac13\big]\big) \cong \bZ/2$; 
 \item[$4)$] $\H^4\big(\Fi_{24}'; \bZ\big[\frac13\big]\big)$ has order $2$ or~$4$. 
 \end{enumerate}
\end{Theorem}

\begin{proof}
 Lemma~\ref{large primes} handles all primes $p\geq 5$ except for $\H^4(\Fi_{24}')_{(7)}$. But the $7$-Sylow in $\Fi_{24}'$ is inside a copy of Held's group $\mathrm{He}$, and $\H^4(\mathrm{He})_{(7)} = 0$ by Theorem~\ref{thm:he}.

 We now inspect the LHS spectral sequences for the extensions $2^{\lceil N/2\rceil - 1}.\rM_N \subseteq \Fi_N$.
 The three cases are:
 \begin{gather*}
 \begin{array}{ccccc}
 \begin{array}{ccccc}
 0 \\
 0 & 0 & 2 \\
 0 & 2 & 0 \\
 0 & 0 & 0 & 0 & 0 \\
 \bZ & 0 & 0 & 12 & 0 \\
 \end{array}
 &\qquad&
 \begin{array}{ccccc}
 0 \\
 0 & 2 & 2 \\
 0 & 0 & 0 \\
 0 & 0 & 0 & 0 & 0 \\
 \bZ & 0 & 0 & 0 & 0
 \end{array}
 &\qquad&
 \begin{array}{ccccc}
 0 \\
 0 & 2 & \\
 0 & 0 & 0 \\
 0 & 0 & 0 & 0 & 0 \\
 \bZ & 0 & 0 & 0 & 12
 \end{array}
 \\ \\
 2^{10}:\rM_{22} \subseteq \Fi_{22}
 &&
 2^{11}.\rM_{23} \subseteq \Fi_{23}
 &&
 2^{11}.\rM_{24} \subseteq \Fi_{24}
 \end{array}
 \end{gather*}
 The first proves claim (1) and provides the upper bound for claim (2). The second provides the upper bound for claim (3).

 Let $\alpha \in \H^4(\rM_{24})_{(2)} \cong \bZ/4$ denote a generator, and let $\tilde\alpha$ denote its pullback to $2^{11}.\rM_{24}$. Consider the conjugacy classes $4\rb$ and $4\rc$ in $\rM_{24}$. The formula given in \cite[Section~3.3]{GPRV} provides $\alpha|_{\langle 4\rb\rangle } = 0$ whereas $\alpha|_{\langle 4\rc\rangle}$ has order $4$ in $\H^4(\langle 4\rc\rangle)$. These classes admit preimages in $2^{11}.\rM_{24}$ living in conjugacy classes $8\re$ and $8\ra$ respectively, and so $\tilde\alpha|_{8\re} = 0$ whereas $\tilde\alpha|_{8\ra}$ has order $2$. (The pullback map $\H^4(C_4) \to \H^4(C_8)$ along a surjection of cyclic groups $C_8 \to C_4$ has image of order~$2$.) Both $8\re$ and $8\ra$ fuse in $\Fi_{24}'$ to class $8\ra$, and so $\tilde\alpha$ cannot be the restriction of a~cohomology class on $\Fi_{24}'$. Together with the above spectral sequence, we find the upper bound in claim~(4).

 All that remain are the lower bounds. Let $\omega^\natural$ denote the ``gauge anomaly'' of the Monster CFT, studied in~\cite{JFmoonshine}; cf.\ Section~\ref{sec:monster}. The McKay--Thompson series for class $4\rb$ in the Monster group~$\rM$ has a nontrivial multiplier (of order~$2$), and so $\omega^\natural|_{\langle 4\rb\rangle}$ is nonzero. But $4\rb$ meets $2\Fi_{22} \subseteq \Fi_{23} \subseteq 3\Fi_{24}'$, and so $\omega^\natural$ restricts with order at least $2$ to all of these groups.
\end{proof}

\begin{Theorem} $\H^4(\Fi_{22}) = \{0\}$ and $\H^4(3\Fi_{22}) \cong \bZ/3$.
\end{Theorem}
Combined with Theorem~\ref{thm.fi}(2), we find that $\H^4(6\Fi_{22})$ has order either~$6$ or~$12$.
\begin{proof}
 Given Theorem~\ref{thm.fi}(1), we must only compute $\H^4(\Fi_{22})_{(3)}$ and $\H^4(3\Fi_{22})_{(3)}$. But the $3$-Sylow in $\Fi_{22}$ is contained in a maximal subgroup isomorphic to $\Omega_7(3)$, and $\H^4(\Omega_7(3))_{(3)} = 0$ by Lemma~\ref{lemma:O73}. By Lemma~\ref{lemma:cokers}, the inclusions $\H^4(\Omega_7(3))_{(3)} \to \H^4(3\Omega_7(3))_{(3)}$ and $\H^4(\Fi_{22})_{(3)} \to \H^4(3\Fi_{22})_{(3)}$ have the same cokernel; this cokernel is $\bZ/3$ by
 Corollary~\ref{cor:3o73}.
\end{proof}

\subsection{Monster} \label{sec:monster}

The main result of \cite{JFmoonshine} is that $\H^4(\rM)$ contains a class $\omega^\natural$, arising as the gauge anomaly of the Moonshine CFT, of exact order $24$.
It is reasonable to conjecture that $\omega^\natural$ generates $\H^4(\rM)$. The calculations in this paper allow us to come close to proving that conjecture:

\begin{Theorem}
 $\H^4(\rM) \cong \bZ/24 \oplus X$, where the $\bZ/24$ summand is generated by $\omega^\natural$ and where the order of $X$ divides $4$.
\end{Theorem}

\begin{proof}
 The primes $p=11$ and $p\geq 17$ are covered by Lemma~\ref{large primes}. To calculate $\H^4(\rM)$, we must study the $p$-Sylows for $p=2,3,5,7,13$. For these $p$, the $p$-Sylow is contained in the normalizer~$N(p\rb)$ of an element of conjugacy class $p\rb$. These normalizers are all of shape
 \begin{gather*} N(p\rb) = p^{1 + m}.J , \end{gather*}
 where $m = 24/(p-1)$ and $J \subseteq \Co_1$. Specifically
 \begin{gather*} \begin{array}{c|c}
 p & J \\ \hline
 2 & \Co_1 \\
 3 & 2\Suz.2 \\
 5 & 2\rJ_2.4 \\
 7 & 3 \times 2S_7 \\
 13 & 3 \times 4S_4 \end{array} \end{gather*}
 The extension $p^{1 + m}.J$ splits for $p\geq 5$. When $p=3$, $3^{1+12}.2.\Suz.2$ does not split, but the quotient $3^{12}:2\Suz.2$ does split. When $p=2$, the quotient $2^{24}\cdot \Co_1$ does not split.

 The center of the group $J$ in all cases has order $p-1$, and acts by a faithful central character on $p^m$. Applying Lemma~\ref{lemma:central character},
 we find that $\H^i\big(J; \H^j\big(p^{1+m}\big)\big) = 0$ for $0 < j < p$. Combined with the HAP calculation of $\H^4(2\rJ_2)$ from Section~\ref{sub:J2}, the known value $\H^4(3\times 2S_7)_{(7)} = 0$ (which follows by the methods of Lemma~\ref{large primes}), and the trivial result $\H^4(3 \times 4S_4)_{(13)}$, we find that $\H^4\big(\rM; \bZ\big[\frac16\big]\big) = 0$.

 At the prime $3$, central character considerations imply that the only potentially nonzero entries of total degree $4$ on the LHS spectral sequence for $3^{1+12}.2\Suz.2$ are, in the notation of Lemma~\ref{lem:extraspecial-odd}:
 \begin{gather*} \H^0\big(2\Suz.2; \Sym^2\big(3^{12}\big)\big), \qquad \H^1\big(2\Suz.2; \Alt^2\big(3^{12}\big)_\omega\big), \qquad \H^4(2\Suz.2; \bZ)_{(3)}. \end{gather*}
 The first vanishes because $3^{12}$ is antisymmetrically but not symmetrically self-dual as a $2.\Suz$ module. Actually, as a $2\Suz.2$ module, $3^{12}$ is not self-dual at all: the symplectic pairing changes by a sign via the surjection $2\Suz.2 \to 2$. The last vanishes by Theorem~\ref{thm:suz}.

 Therefore $\H^1\big(2\Suz.2; \Alt^2\big(3^{12}\big)_\omega\big)$ gives an upper bound for $\H^4(\rM)_{(3)}$. The group $2\Suz.2$ is too large for Cohomolo to handle directly, but its $3$-Sylow-containing maximal subgroup of shape $3^5:(\rM_{11} \times 2)$ is not, and Cohomolo computes
 \begin{gather*} \H^1\big( 3^5:(\rM_{11} \times 2); \Alt^2\big(3^{12}\big)_\omega\big) \cong 3^1. \end{gather*}
 This provides an upper bound for $\H^1\big(2\Suz.2; \Alt^2\big(3^{12}\big)_\omega\big)$, and so
 \begin{gather*} \H^4(\rM)_{(3)} \leq 3. \end{gather*}
 On the other hand, \cite[Lemma 3.2.4]{JFmoonshine} shows that $\H^4(\rM)_{(3)} \geq 3$.

 The $p=2$ part of $\H^4(\rM)$ is more subtle.

We first claim that the $\bZ/8 \subseteq \H^4(\rM)_{(2)}$ generated by $\omega^\natural$ is a direct summand. Equivalently, we claim that $\omega^\natural$ is not divisible by~$2$. Consider the subgroup $N(3\rb) = 3^{1+12}.2.\Suz.2 \subseteq \rM$. Since its quotient $3^{12}:2\Suz.2$ splits, we can find a copy of $6\Suz \subseteq 6\Suz.2 \subseteq \rM$. The central $3 \subseteq 6\Suz$ is generated by class $3\rb$ by construction, and the central $2 \subseteq 6\Suz \subseteq \rM$ is of class $2\rb$. It follows that this $6\Suz$ is (conjugate to) a subgroup of the normalizer $N(2\rb) = 2^{1+12}\cdot \Co_1$, where it lives over a copy of $3\Suz \subseteq \Co_1$. As observed in the proof of Theorem~\ref{thm:suz}, the corresponding $6\Suz \subseteq 2\Co_1$ contains the group $2D_8$ used in \cite{jft}; thus the $6\Suz \subseteq \rM$ contains a $2D_8$ such that the central element is of class $2\rb \in \rM$. This is the $2D_8 \subseteq \rM$ used in \cite[Section~3.3]{JFmoonshine} to show that the order of $\omega^\natural$ is divisible by $8$. It follows in particular that the order of $\omega^\natural|_{6\Suz}$ is divisible by~$8$. But $\H^4(6\Suz)_{(2)} = 8$ by Theorem~\ref{thm:suz}, and so $\omega^\natural|_{6\Suz}$ is not divisible by $2$, proving the claim.

 The last step of the proof is to study the LHS spectral sequence for the extension $2^{1+24}\cdot\Co_1$:
 \begin{gather*} \begin{array}{ccccc}
 \H^0\big(\Co_1; \H^4\big(2^{1+24}\big)\big) \\
 \H^0\big(\Co_1; \H^3\big(2^{1+24}\big)\big) & \H^1\big(\Co_1; \H^3\big(2^{1+24}\big)\big) \\
 \H^0\big(\Co_1; \H^2\big(2^{1+24}\big)\big) & \H^1\big(\Co_1; \H^2\big(2^{1+24}\big)\big) & \H^2\big(\Co_1; \H^2\big(2^{1+24}\big)\big) \\
 0 & 0 & 0 & 0 & 0 \\
 \bZ & 0 & 0 & 2 & 4
 \end{array} \end{gather*}
 The $2$ in the bottom row is $\H^3(\Co_1)$, and the $4$ in the bottom row is $\H^4(\Co_1)_{(2)}$; the latter result is due to~\cite{jft}.

 From Section~\ref{sub.extraspecial 2group}, we know
 \begin{gather*} \H^2\big(2^{1+24}\big) \cong 2^{24}, \qquad \H^3\big(2^{1+24}\big) \cong \Alt^2\big(2^{24}\big)/2 \cong 2^{275},\end{gather*}
 while
 \begin{gather*} \H^4\big(2^{1+24}\big) \cong 2^{24}.\Alt^2\big(2^{24}\big).\big(\Alt^3\big(2^{24}\big)/2^{24}\big).2 \cong 2^{2300}.2 \end{gather*}
 has exponent $4$.

 Without much difficulty, GAP can compute
 \begin{gather*} \H^0\big(\Co_1; 2^{24}\big) = \H^0\big(\Co_1; \Alt^2\big(2^{24}\big)/2\big) = \H^0\big(\Co_1; \Alt^3\big(2^{24}\big)\big) = 0.\end{gather*}
 and
 \begin{gather*} \H^0\big(\Co_1; \Alt^2\big(2^{24}\big)\big) \cong \bZ/2. \end{gather*}

 The groups $\H^i\big(\Co_1; 2^{24}\big)$ for $i=1,2$ were calculated by Derek Holt and reported in \cite[Lemma~1.8.8]{Ivanov09}. They are
 \begin{gather*} \H^1\big(\Co_1; 2^{24}\big) = 0, \qquad \H^2\big(\Co_1; 2^{24}\big) \cong \bZ/2. \end{gather*}
 A presentation for $\Co_1$ is given in~\cite{MR886429}. Using it, \cite[Section~3.5]{JFmoonshine} calculates
 \begin{gather*} \H^1\big(\Co_1; \Alt^2\big(2^{24}\big)/2\big) \cong \bZ/2.\end{gather*}

 These calculations almost fill in the $E_2$ page. The remaining missing ingredient is $E_2^{04} = \H^0\big(\Co_1; \H^4\big(2^{1+24}\big)\big)$.
 Using the above calculations together with the long exact sequence for cohomology with values in an extension, one finds:
 \begin{gather*} \H^0\big(\Co_1; 2^{24}.\Alt^2\big(2^{24}\big).\big(\Alt^3\big(2^{24}\big)/2^{24}\big)\big) = \H^0\big(\Co_1; 2^{2300}\big) \cong \bZ/2. \end{gather*}
 It follows that $\H^0\big(\Co_1; \H^4\big(2^{1+24}\big)\big)$ is isomorphic to either $\bZ/2$ or $\bZ/4$. We suspect the former, but were unable to compute it.

 Although we do not know whether $E_2^{04} = 2$ or $4$, we claim that $E_\infty^{04} = 2$. To prove this we quote two facts. First, according to \cite[Section~3.5]{JFmoonshine}, $\omega^\natural|_{2^{1+24}}$ has exact order~$2$, and so provides an element of order $2$ in $E_\infty^{04}$. Second, we showed above that $\omega^\natural$ is not divisible by $2$ in $\H^4(\rM)$. Since the map $\H^4(\rM) \to \H^4\big(2^{1+24}.\Co_1\big)$ is an inclusion onto a direct summand, it follows that $\omega^\natural|_{2^{1+24}.\Co_1}$ is not divisible by~$2$. But $\H^4\big(2^{1+24}.\Co_1\big)$ surjects onto $E_\infty^{04}$, and sends~$\omega^\natural$ to a~nonzero value. So the image of~$\omega^\natural$ in~$E_\infty^{04}$ cannot be divisible by $2$, proving that $E_\infty^{04} \neq 4$.

 All together, we find an $E_\infty$ page of the following form
 \begin{gather*} \begin{array}{ccccc}
 2 \\
 0 & \leq 2 \\
 0 & 0 & \leq 2 \\
 0 & 0 & 0 & 0 & 0 \\
 \bZ & 0 & 0 & 2 & 4
 \end{array} \end{gather*}
 It follows that $\H^4(\rM)_{(2)}$ has order at most $32$, completing the proof.
\end{proof}

\subsection*{Acknowledgements}

This project began when John Duncan asked us to compute $\H^4(\mathrm{O'N};\bZ)$.
The referees provided detailed and valuable corrections.
We also thank Graham Ellis, Jesper Grodal, Geoff Mason, and Robert Wilson for numerous comments and hints about finite groups and cohomology.
This work was supported by the US NSF grant DMS-1510444 and by the Government of Canada through the Department of Innovation, Science and Economic Development Canada and by the Province of Ontario through the Ministry of Research, Innovation and Science.

\pdfbookmark[1]{References}{ref}
\LastPageEnding

\end{document}